\documentclass{amsart}

\usepackage{cancel}

\usepackage{amssymb}
\usepackage{amsthm}
\usepackage{xcolor}
\usepackage{mathtools}
\usepackage{hyperref}
\usepackage{picture}
\usepackage{mathrsfs}
\usepackage[normalem]{ulem}

\usepackage{float}

\newcommand{\B}{\mathcal{B}}
\newcommand{\E}{\mathcal{FK}}
\newcommand{\Tau}{\mathcal{T}}
\newcommand{\PP}{\mathbb{P}}
\newcommand{\kk}{K}
\newcommand{\Span}{\operatorname{span}}

\newcommand{\Aff}{\operatorname{Aff}}
\newcommand{\charac}{\operatorname{char}}
\newcommand{\rank}{\operatorname{rank}}
\newcommand{\ord}{\operatorname{ord}}
\newcommand{\pr}{\operatorname{pr}}
\newcommand{\wt}{\operatorname{wt}}
\newcommand{\GL}{\operatorname{GL}}

\makeatletter
\numberwithin{equation}{section}
\numberwithin{figure}{section}
\numberwithin{table}{section}

\newtheorem{thm}{Theorem}[section]
\newtheorem{lem}[thm]{Lemma}
\newtheorem{prop}[thm]{Proposition}
\newtheorem{prob}[thm]{Problem}
\newtheorem{cor}[thm]{Corollary}

\newtheorem{defn}[thm]{Definition}

\newtheorem{rem}[thm]{Remark}
\newtheorem{exa}[thm]{Example}

\makeatother

\title{Point modules of Nichols algebras over non-abelian groups}

\author{Be'eri Greenfeld} \address{Department of Mathematics and Statistics, Hunter College and CUNY Graduate Center, 695 Park Avenue, New York NY 10065, USA} \email{beeri.greenfeld@hunter.cuny.edu}

\author{L. Vendramin} \address{Department of Mathematics and Data Science, Vrije Universiteit Brussel, Pleinlaan 2, 1050 Brussel, Belgium}
\email{Leandro.Vendramin@vub.be}

\author{S. Venkatesh} \address{University of Washington, Department of Mathematics, Box 354350, C-138 Padelford Hall
Seattle, WA 98195-4350 USA}
\email{samnsid7@uw.edu}

\begin{document}

\begin{abstract}
Nichols algebras are among the most important and mysterious objects in quantum algebra. We continue the systematic program of studying them through their point modules.

We study truncated point modules of Nichols algebras associated with racks and $2$-cocycles over non-abelian groups. We determine the truncated point spaces for all currently known finite-dimensional Nichols algebras arising from simple Yetter--Drinfeld modules over groups. In particular, for all such algebras the space of $3$-truncated point modules is empty. We further find all truncated point modules for other important families, including affine racks equipped with the constant cocycle $-1$, and for the non-trivial indecomposable rack of size three with the constant cocycle $\zeta_3$.
\end{abstract}

\maketitle 

\section{Introduction}

\subsection{Nichols algebras}
An important part of the recent theory of Hopf algebras relies on the
classification of finite-dimensional pointed Hopf algebras. The successful
approach used to tackle this challenging problem is based on the lifting method
of Andruskiewitsch and Schneider~\cite{zbMATH01247621}. The core of this method is related to
the study of certain braided Hopf algebras known as Nichols algebras.

Let $K$ be a field. A \emph{braided vector space} is a finite-dimensional $K$-vector space $V$ equipped with a braiding operator $c\colon V\otimes_{K} V\rightarrow V\otimes_{K} V$ , namely, an invertible map satisfying $(I\otimes c)(c\otimes I)(I\otimes c)=(c\otimes I)(I\otimes c)(c\otimes I)$. 

To every braided vector space $(V, c)$ we associate a certain (braided Hopf) graded 
algebra 
known as the \emph{Nichols algebra} of $(V, c)$. It is defined as
\[
\mathcal{B}(V,c)=K\oplus V\oplus\bigoplus_{n\geq2}V^{\otimes n}/\ker\mathcal{S}_n,
\]
where the \emph{quantum symmetrizer} $\mathcal{S}_n$ is defined 
as 
\begin{align*}
    &\mathcal{S}_2 = I+c_1,\\
    &\mathcal{S}_3 = I+c_1+c_2+c_1c_2+c_2c_1+c_1c_2c_1,\\
    &\quad\vdots \\
    &\mathcal{S}_{n+1} = (I\otimes\mathcal{S}_n)(I+c_1+c_1c_2+\cdots+c_1c_2\cdots c_n).
\end{align*}
Here $c_i(v_1\otimes \cdots \otimes v_n)=v_1\otimes \cdots \otimes c(v_i \otimes v_{i+1}) \otimes \cdots \otimes v_n$.
Although a Nichols algebra is associated with every braided vector space, the
case of braided vector spaces arising from groups is particularly rich and
interesting. Let $G$ be a finite group and let $\mathcal{C}\subseteq G$ be a union of conjugacy classes. Then for a certain collection of non-zero scalars 
$\{q_{gh}:g,h\in\mathcal{C}\}$, 
the map 
\[        
c(g\otimes h)=q_{gh}(ghg^{-1})\otimes g,
\]
extends to a braiding on $V=K[\mathcal{C}]$. The resulting braided vector spaces can be realized as
Yetter--Drinfeld modules
over $K[G]$ (namely, simultaneously compatible modules and comodules over $K[G]$).

A crucial step in the lifting method of Andruskiewitsch and Schneider is 
the \emph{classification of finite-dimensional Nichols algebras}. This fundamental problem
has attracted significant attention over the past 25 years; see
\cite{zbMATH06797653} and the references therein. While the problem
remains open in full generality, several breakthroughs have been made toward its resolution. For
example, it is completely solved for abelian groups~\cite{zbMATH06224213,zbMATH05027328,zbMATH05376870}, but remains wide open over non-abelian groups. 
When the Yetter--Drinfeld module is semisimple but not simple then, under some
natural and mild assumptions, the classification of finite-dimensional Nichols
algebras is known~\cite{zbMATH06694062,zbMATH06736624}. However, the situation is entirely different for simple Yetter--Drinfeld modules. 
Although there is an explicit description of Nichols algebras using the quantum symmetrizer, it is in general very difficult to determine whether a given Nichols algebra is finitely presented, or to find explicit relations in cases where it is, or even whether it is finite-dimensional or not. This makes it nearly impossible to apply standard methods from combinatorial algebra.
Despite recent progress made for Nichols algebras over solvable groups~\cite{arXiv:2411.02304,zbMATH07837859} or over various special classes of simple groups~\cite{zbMATH07796386,zbMATH05903856,zbMATH05869018}, the 
general problem still seems somewhat out of reach.


In~\cite{zbMATH05942621}, the authors studied Nichols algebras over the symmetric group $\mathbb{S}_4$ and, in the process, provided a precise description of the space of quadratic relations. Similarly, while investigating cubic relations, the authors of \cite{zbMATH06050325} introduced \emph{braided racks}, which turns out to be highly useful in this context.

Currently, we know of only a few examples of finite-dimensional Nichols algebras over groups; see Table~\ref{tab:examples}. Whether this list is complete remains an open question. Identifying patterns among these examples will naturally aid in the classification of finite-dimensional Nichols algebras over groups.


\subsection{Noncommutative projective geometry}
Let us take a noncommutative (projective) geometric approach. Let $A=\bigoplus_{i=0}^{\infty} A_i$ be a connected graded $K$-algebra ($A_0=K$), finitely generated by degree $1$ elements. A point module over $A$ is a graded, cyclic $A$-module generated in degree zero and with one-dimensional homogeneous components, namely, with Hilbert series 
\[
H(t)=(1-t)^{-1}=1+t+t^2+\cdots.
\]
Point modules of noncommutative graded algebras take the place of points on projective algebraic varieties in classical algebraic geometry and, indeed, if $A$ is commutative then its point modules are parametrized by the proj construction.

The parameter space of point modules over $A$, denoted $\Gamma(A)$, can be realized as the inverse limit of a system of projective schemes
$$\Gamma(A)=\varprojlim_n\  \Gamma_n(A),\ \ \ \cdots \rightarrow \Gamma_2(A)\rightarrow \Gamma_1(A) \rightarrow \Gamma_0(A)$$
where $\Gamma_n(A)$ is the moduli scheme of $n$-truncated point modules over $A$: these are graded, cyclic $A$-modules, generated in degree zero and with one-dimensional homogeneous components up to degree $n$; namely, with Hilbert series $H(t)=1+t+\cdots+t^n$. Artin and Zhang \cite{AZ} proved that, for strongly Noetherian algebras, the above system stabilizes and hence $\Gamma(A)$ is a projective scheme. However, in the general non-Noetherian setting, $\Gamma(A)$ need not even be an algebraic stack, see e.g. \cite{BG,Rogalski,Smith}. By a (slight) abuse of notation, we think of $\Gamma_n(A)$ as its set of closed points over a given field, as we are more interested in the existence of (truncated) point modules themselves rather than the scheme structure of the moduli space. Namely, $\Gamma_n(A)$ stands in a bijection with the set of $n$-truncated point modules over $A$.

Point modules and truncated point modules have gained significant importance in noncommutative projective algebraic geometry, as they allow us to associate geometric objects to noncommutative algebraic structures and thus study the latter by geometric means, see \cite{ATV, ATV2, AZ2, AZ, RRZ, RZ, Smith}. Furthermore, this approach turned out to be extremely valuable even for purely ring theoretic purposes. A notable instance is the solution of Sierra and Walton \cite{SW} of the Dean--Small problem on the Noetherianity of the enveloping algebra of the Witt Lie algebra. A systematic extension of the theory of point modules over noncommutative algebras to the setting of line modules can be found in \cite{SheltonVancliff2002a,SheltonVancliff2002b}.

An evident observation is that a graded algebra admitting a point module is infinite-dimensional. Therefore, to prove that a given algebra is infinite-dimensional it suffices to prove that it --- or any of its graded subalgebras --- admit a point module. Motivated by this, we can think of
$$p(A) := \sup \{n\geq 0\ |\ \Gamma_n(A)\neq \emptyset\}$$
as a measure of `largeness' of $A$, and in particular, if $p(A)=\infty$ then $A$ is infinite-dimensional. (Nevertheless, even some very well behaved graded algebras can lack point modules, see \cite{Vancliff2024}.) Even when $A$ is finite-dimensional, it is interesting to study $p(A)$ and the associated spaces of truncated point modules over $A$. With this in mind, a natural approach to understand the structure of Nichols algebras is to study their (truncated) point modules. 

\begin{prob} 
\label{prob:main}
    Compute $\Gamma_n(\B(X,q))$ for finite racks $X$ and $2$-cocycles $q$.
\end{prob}

Problem \ref{prob:main} is especially interesting as a strategy to prove that certain Nichols algebras are infinite-dimensional, or at least, to give non-trivial lower bounds on their Hilbert series (or Hilbert polynomials).

This challenge has been proposed in \cite{An2004} and developed in \cite{ABFF2023}. Motivated by this, in \cite{GMSW}, the invariant $p(-)$ has been computed for Fomin--Kirillov algebras (which appear as an important special case of Nichols algebras over the symmetric groups, and for which 
one does not even know when they are finite-dimensional).

\subsection{Main results}

An important class of racks is the class of \emph{braided} racks. Following \cite{zbMATH06050325}, we say that a rack $X$ is braided if it is a quandle (that is, $x\triangleright x = x$) and for all $x, y\in X$ either $x\triangleright (y\triangleright x)=y$ or $x\triangleright y = y$. These naturally appear in the study of Nichols algebras with few low-degree relations. All our racks are in fact unions of conjugacy classes of a group.

\begin{prop}
    Let $X$ be a braided rack and let $q\equiv -1$. Then $\Gamma_2\left(\B(X,q)\right)=\emptyset$.
\end{prop}

Another important class of racks arising from solvable groups is of \emph{affine racks}, see definition and examples in Section \ref{sec:affine racks}. Some of the known finite-dimensional Nichols algebras arise from affine racks but, in general, these racks can yield infinite-dimensional Nichols algebras.

\begin{thm} \label{thm:affine}
Let $X=\Aff(p,\alpha)$ be an affine rack and let $\B(X,q)$ be the Nichols algebra associated with $X$ and with the cocycle $q\equiv -1$ over a field of characteristic zero. Then $\Gamma_2(\B(X,q))=\emptyset$.
\end{thm}

Equipped with these results, we turn to analyze the spaces of truncated point modules over the full list of known finite-dimensional Nichols algebras arising from simple Yetter--Drinfeld modules over groups (see Table~\ref{tab:examples}).


\begin{thm} 
\label{thm:table}
For finite-dimensional Nichols algebras over the known simple Yetter-Drinfeld modules, we have truncated point spaces as given in Table \ref{tab:examples}. In particular, for all of them $\Gamma_3=\emptyset$.
\end{thm}

\begin{table}[H] 
    \centering
    \begin{tabular}{|l|l|l|l|l|l|}
        \hline 
        Rack & Rank & Dimension & $2$-cocycle & Characteristic & Point spaces \tabularnewline
        \hline $(12)^{\mathbb{S}_3}$ & 
        3 & $12$ & $-1$ & & $\Gamma_2=\emptyset$ \\
        \hline
        $(12)^{\mathbb{S}_3}$ & 3 & $432$ & $\zeta_3$ & $2$ & $\Gamma_2\neq \emptyset,\ \Gamma_3=\emptyset$ \\
        \hline
        $\mathcal{T}$ & 4 & $36$ & 1 & $2$ & $\Gamma_2 = \emptyset$ \\
        \hline
        $\mathcal{T}$ & 4 & $72$ & $-1$ & $\ne 2$ & $\Gamma_2 = \emptyset$ \\
        \hline
        $\mathcal{T}$ & 4 & $5184$ & $\psi$ & & $\Gamma_2\neq \emptyset,\ \Gamma_3=\emptyset$ \\
        \hline
         $\Aff(5,2)$ & 5 & $1280$  & $-1$ & & $\Gamma_2=\emptyset$ \\ \hline
        $\Aff(5,3)$ & 5 & $1280$  & $-1$ & & $\Gamma_2=\emptyset$ \\ \hline
        $(12)^{\mathbb{S}_4}$ & 6 & $576$  & $\chi$ & & $\Gamma_2=\emptyset$ \\ \hline
        $(12)^{\mathbb{S}_4}$ & 6 & $576$  & $-1$ & & $\Gamma_2=\emptyset$ \\ \hline
        $(1234)^{\mathbb{S}_4}$ & 6 & $576$  & $-1$ & & $\Gamma_2=\emptyset$ \\ \hline
        $\Aff(7,3)$ & 7 & $326592$ & $-1$ & & $\Gamma_2=\emptyset$ \\
        \hline
        $\Aff(7,5)$ & 7 & $326592$ & $-1$ &  & $\Gamma_2=\emptyset$ \\
        \hline
        $(12)^{S_5}$ & 10 & $8294400$  & $\chi$ & & $\Gamma_2=\emptyset$ \\ \hline
        $(12)^{S_5}$ & 10 & $8294400$  & $-1$ & & $\Gamma_2=\emptyset$ \\ \hline
    \end{tabular}
    \caption{Finite-dimensional Nichols algebras over all known simple Yetter--Drinfeld modules.} \label{tab:examples}
\end{table}

We then move to cocycles defined by higher roots of unity, a much harder situation in which even less is known even for the smallest non-trivial indecomposable rack $X=(12)^{\mathbb{S}_3}\cong \Aff(3,-1)$. For cocycle $q\equiv -1$, $\B(X,q)\cong \E_3$; with $q\equiv \zeta_3\in \mathbb{F}_4$, $\B(X,q)$ is covered by Proposition \ref{prop:cubic roots}, and in particular $\Gamma_2(\B(X,q))\neq \emptyset$ while $\Gamma_3(\B(X,q))=\emptyset$. If the cocycle is $q\equiv \zeta_3\in K$ and $\text{char}(K)=0$ then $\B(X,q)$ is infinite-dimensional (see \cite{zbMATH07837859}) and, in contrast with the case of $\text{char}(K)=2$, we get a much more complicated behavior of the truncated point spaces.

\begin{thm} \label{prop:zeta3}
Let $X=\Aff(3,-1)$ and $q\equiv \zeta_3\in K$ where $\text{char}(K)=0$. Then: $$\Gamma_n\left(\B(X,q)\right) = \begin{cases}
    \mathbb{P}^2\times \mathbb{P}^2 & \text{if} \ n=2, \\ 90\ \text{points} & \text{if} \ n=3, \\ 
    36\ \text{points} & \text{if} \ n=4, \\ 
    \emptyset & \text{if} \ n\geq 5.
\end{cases}$$
\end{thm}


Finally, we discuss truncated `thick' point modules over the Fomin--Kirillov algebra $\E_3$ and demonstrate that they form a richer space than truncated point modules. 
Chan \cite{Chan_fat} showed that thick point modules play an important role in noncommutative projective algebraic geometry, and that their parametrizing stacks can be used to construct canonical morphisms in a similar vein to the Rogalski--Zhang morphisms \cite{RZ}; see also \cite{Goetz2007}.
For a graded algebra $A$, let $\Gamma_{c,d}(A)$ be the moduli space of graded $A$-modules generated in degree zero with Hilbert polynomial $c+ct+\cdots+ct^d$ (truncated thick point modules of width $c$ and degree $d$). Thus $\Gamma_{1,d}=\Gamma_d$.

\begin{thm} 
\label{thm:22}
Let $K$ be a field and let $\E_3$ be the Fomin--Kirillov algebra on three vertices. We have:
$$\#\Gamma_{2,2}(\E_3) =  \begin{cases}
        4 & \text{if } \charac(K)=3\text{ or }\charac(K)\neq 3\ \text{and}\ \zeta_3\notin K,\\  
        6 & \text{if }\ \charac(K)\neq 3\ \text{and}\ \zeta_3\in K.
    \end{cases}$$
For $d>2$, we have $\Gamma_{2,d}(\E_3)=\emptyset$.
\end{thm}

\medskip

\emph{Conventions.} All algebras are associative and unital, and graded algebras are assumed to be connected graded (that is, their degree zero components are the base field) and generated in degree one. 
We denote the complement of a subset $S\subseteq T$ by $S^c$ when the ambient set is clear.

\section*{Acknowledgments}

We are thankful to James J. Zhang for the original idea of studying Fomin--Kirillov algebras using point modules and thick point modules, and for inspiring related discussions. 
This work was partially supported by OZR3762
of Vrije Universiteit Brussel, FWO V507824N, 
and FWO Weave Senior Research Project G043326N.

\section{Racks and their Nichols algebras}





 A \emph{rack} is a set $X$ together with a binary operation
$\triangleright\colon X\times X\to X$ such that the maps $\varphi_x\colon
X \to X$, given by $y\mapsto x\triangleright y$, are bijective
for each $x\in X$, and $x\triangleright (y\triangleright z)=(x\triangleright
y)\triangleright (x\triangleright z)$ for all $x,y,z\in X$.

\begin{exa} \label{exa:conjugacy}
	Let $G$ be a group and $X$ be a union of conjugacy classes. Then $X$ with
	$x\triangleright y=xyx^{-1}$ for all $x,y\in X$ is a rack. 
\end{exa}

\begin{defn}
A rack $(X,\triangleright)$ is of \emph{conjugacy type} if it is isomorphic to a union of conjugacy classes of a group as in Example \ref{exa:conjugacy}.
\end{defn}
By convention, all the racks of this paper will be of conjugacy type.

\begin{rem} Let $(X,\triangleright)$ be a rack of conjugacy type. Then for all $x\in X$, $x\triangleright x=x$ (a rack satisfying this property is also known as a \emph{quandle}), and $x\triangleright y=y$ if and only if $y\triangleright x = x$ (indeed, in a group, both are equivalent to $x,y$ commuting with each other).
\end{rem}

The \emph{inner group} of $X$ is the group
$\mathrm{Inn}(X)=\langle\varphi_x:x\in X\rangle$. A rack is called
\emph{indecomposable} if $\mathrm{Inn}(X)$ acts transitively on $X$.  The
\emph{enveloping group} of $X$ is the group
\[
G_X=\langle x_i\,:x_ix_j=x_{i\triangleright j}x_i\text{ for all $i,j\in X$}\rangle.
\]

The enveloping group has the following \emph{universal property}: For any
group $G$ and any map $f:X\to G$ satisfying $f(x\triangleright
y)=f(x)f(y)f(x)^{-1}$ there exists a unique group homomorphism $g:G_X\to G$
such that $f=g\circ\partial$, where $\partial:X\to G_X$, $i\mapsto x_i$.

Let $X$ be a rack and $V=K X$ be the vector space (over the field $K$)
with basis $X$.  Let 
\[
q\colon X\times X\to K^{\times},\quad (x,y)\mapsto q_{x,y},
\]
be a map 
and consider the map $c\in\GL(V\otimes V)$ given by $c(x\otimes
y)=q_{x,y}(x\triangleright y)\otimes x$. Then $(V,c)$ is a braided vector
space if and only if $q$ is a \textbf{$2$-cocycle} of $X$, that is 
\begin{align*}
		q_{x,y\triangleright z}q_{y,z}=q_{x\triangleright y,x\triangleright z}q_{x,z}
	\end{align*}
for all $x,y,z\in X$. The braided vector space $(V,c)$ is said to be of
type $(X,q)$. 
These 2-cocycles arise in the context of rack cohomology. For further details, see~\cite{Andruskiewitsch2003}.

\begin{exa}	\label{exa:2.2}
	Let $\Tau=(123)^{\mathbb{A}_4}$ be the rack associated with the conjugacy class
	of $(123)$ in the alternating group $\mathbb{A}_4$. The map $\psi\colon \Tau\times
	\Tau\to K^\times$ given by
	\begin{equation}
		\label{eq:2cocycle}
		\begin{array}{c|cccc}
			& (243) & (123) & (134) & (142)\\
			\hline (243) & \zeta & \zeta & \zeta & \zeta\\
			(123) & \zeta & \zeta & -\zeta & -\zeta\\
			(134) & \zeta & -\zeta & \zeta & -\zeta\\
			(142) & \zeta & -\zeta & -\zeta & \zeta
		\end{array}
	\end{equation}
	where $\zeta\in K$ is a primitive cubic root of unity, is a $2$-cocycle of
	$\Tau$. There are three known finite-dimensional Nichols algebras
    over this rack, namely
    \begin{enumerate}
        \item for fields of characteristic different from two, $\B(\Tau,-1)$ of dimension 72; 
        \item for fields of characteristic two, $\B(\Tau,-1)$ of dimension 36; and 
        \item for all suitable fields, $\B(\Tau,\psi)$ is of dimension 5184, where $\zeta^2+\zeta+1=0$. 
    \end{enumerate}
\end{exa}

\begin{exa}
	\label{exa:FK}
	Let $n\geq3$ and $X_{n}=(12)^{S_{n}}$.  The map 
	$\chi\colon X_n\times X_n\to K^\times$ given by
	\[
		\chi(\sigma,\tau)=\begin{cases}
		1 & \text{if }\sigma(i)<\sigma(j),\\
		-1 & \text{otherwise,}
	\end{cases}
	\]
	where $\tau=(ij)$, $i<j$, is a $2$-cocycle of $X_{n}$. For $n = 3$, the 2-cocycle $\chi$ 
    is equivalent to the constant $2$-cocycle equal to $-1$.  This does not happen if 
    $n > 3$. There are finite-dimensional Nichols algebras in this setting, namely
    \begin{enumerate}
        \item $\B(X_3,-1)$ of dimension 12;
        \item $\B(X_4,-1)$ and $\B(X_4,\chi)$, both of dimension 576; and 
        \item $\B(X_5,-1)$ and $\B(X_5,\chi)$, both of dimension 8294400. 
    \end{enumerate}
    It is known that for all $n\geq3$, the Nichols algebras 
    $\B(X_n,-1)$ and $\B(X_n,\chi)$ are twist equivalent; see~\cite{zbMATH05896174,zbMATH06204244}. 
    It is still unknown if the two Nichols algebras associated to $X_n$ for $n>5$ 
    are finite-dimensional or not.
\end{exa}

There are two other finite-dimensional Nichols algebras  
associated with symmetric groups. 

\begin{exa}
    Let $K$ be a field of characteristic two containing a primitive third root of unity $\zeta$,
    and $X_3=(12)^{\mathbb{S}_3}$. 
    Then the Nichols algebra $\B(X_3,\zeta)$ has dimension 432.
\end{exa}

\begin{exa}
    Let $X=(1234)^{\mathbb{S}_4}$. Then $\B(X,-1)$ has dimension 576. 
\end{exa}

\begin{exa}
Let $q=p^f$ be a prime power and $\alpha\in \mathbb{F}_q^\times$, $\alpha\neq 1$. Let $(\Aff(q,\alpha),\triangleright)$ be a rack whose underlying set is $X=\{x_g|g\in \mathbb{F}_q\}$ and \[
x_g\triangleright x_h=x_{(1-\alpha)g+\alpha h}.
\]
This rack is called an \emph{affine rack} and 
can be interpreted as a conjugacy class in a group of affine transformations. 
When $q$ is a prime number, it follows from~\cite[Lemma 5.1]{zbMATH02100251} that $2$-cocycles 
are constant, that is $q_{xy}=\gamma\in K^\times$. There are four new finite-dimensional examples
in this setting, namely 
\begin{enumerate}
    \item $\B(\Aff(5,2),-1)$ and $\B(\Aff(5,3),-1)$, both of dimension 1280.
    \item $\B(\Aff(7,3),-1)$ and $\B(\Aff(7,5),-1)$, both of dimension 326592.
\end{enumerate}
Affine racks are $2$-transitive: for each $c\in \mathbb{F}_q^\times, r\in \mathbb{F}_q$, the bijective map $x\mapsto cx+r$ defines an automorphism of $\Aff(q,\alpha)$.
\end{exa}

We have an action of the braid group $$B_n=\langle \sigma_1,\dots,\sigma_{n-1}|\sigma_i\sigma_{i+1}\sigma_i=\sigma_{i+1}\sigma_i\sigma_{i+1},\ \sigma_i\sigma_j=\sigma_j\sigma_i\ \forall |i-j|>1\rangle$$ on $X^{n}$ by $\sigma_i(x_1,\dots,x_n)=(x_1,\dots,x_i\triangleright x_{i+1},x_i,\dots,x_n)$. Thus $c_i \in \GL(V^{\otimes n})$ is given by the action of $\sigma_i$ on pure tensors, twisted by the cocycle $q$. The orbits of $B_n\curvearrowright X^n$ are called \emph{Hurwitz orbits}.

\section{Point modules of Nichols algebras of braided racks}

\subsection{Braided racks}

Following~\cite{zbMATH06050325}, we say that a 
rack $X$ is \emph{braided} if it is a quandle ($x \triangleright x=x$ for all $x\in X$) and, for all $x, y\in X$, either $x\triangleright (y\triangleright x)=y$ or $x\triangleright y = y$. 
A conjugacy-invariant subset of a group forms a braided rack if and only if every two elements of it either commute or satisfy the braid relation $xyx=yxy$.

The following is brought here for the reader's convenience.
\begin{prop} \label{prop:3}
    Let $(X,\triangleright)$ be a braided rack of conjugacy type. Then all Hurwitz orbits of $X^2$ are of size at most $3$.
\end{prop}
\begin{proof}
Fix $(x,y)\in X^2$. If $x\triangleright y=y$ then $y\triangleright x=x$. Thus $\sigma_1(x,y)=(x \triangleright y,x)=(y,x)$ and $\sigma_1(y,x)=(y\triangleright x,y)=(x,y)$, so the Hurwitz orbit of $(x,y)$ is of size at most two.

If $x\triangleright y\neq y$, then $x\triangleright(y\triangleright x)=y$. By \cite[Lemma~1.10]{zbMATH06050325}, $(x\triangleright y)\triangleright x=y$ and $y\triangleright(x\triangleright y)=x$. Indeed, $$(x \triangleright y) \triangleright x = (x \triangleright y) \triangleright (x \triangleright x)=x \triangleright (y \triangleright x)=y$$ and 
$$y \triangleright (x \triangleright y) = (y \triangleright x) \triangleright (y \triangleright y)=(y \triangleright x) \triangleright y = x.$$
Hence
$$(x,y)\xmapsto{\sigma_1} (x \triangleright y, x) \xmapsto{\sigma_1} ((x \triangleright y)\triangleright x, x \triangleright y)=(y,x \triangleright y)\xmapsto{\sigma_1} (y\triangleright (x \triangleright y),y)=(x,y)$$
and the Hurwitz orbit of $(x,y)$ is of size three.
\end{proof}

As mentioned in \cite{zbMATH06050325}, an affine rack $\Aff(q,\alpha)$ is braided if and only if $\alpha^2-\alpha+1=0$. For instance, $\Aff(7,3)$ and $\Aff(7,5)$ are braided, while $\Aff(5,2)$ and $\Aff(5,3)$ are not. 
The racks $\Tau=(123)^{\mathbb{A}_4},\ (12)^{\mathbb{S}_n},\ (1234)^{\mathbb{S}_4}$ are all braided. (In fact, $\Tau\cong \Aff(\mathbb{F}_4,\omega)$ where $\omega$ is a primitive element of $\mathbb{F}_4$.)


\begin{prop} \label{prop:conj S_n braided}
The rack $(12)^{\mathbb{S}_n}$ is braided. The rack $(1234)^{\mathbb{S}_4}$ is braided.
\end{prop}
\begin{proof}
That transpositions either commute or satisfy the braid relation is a well-known fact about the symmetric groups.
For $(1234)^{\mathbb{S}_4}$, notice that it splits into three pairs of $4$-cycles and their inverses. To verify the claim, we may assume, by transitivity of $\mathbb{S}_4\curvearrowright (1234)^{\mathbb{S}_4}$, that one of the permutations is $\sigma=(1234)$. It evidently commutes with $\sigma^{-1}$, and $\langle \sigma \rangle$ acts transitively on the four permutations $(1234)^{\mathbb{S}_4}\setminus \{\sigma,\sigma^{-1}\}$ (indeed, $(1324)^\sigma=(1243),(1324)^{\sigma^2}=(1423),(1324)^{\sigma^{3}}=(1342)$). Hence it suffices to verify that for $\pi=(1324)$ we have
$\sigma\pi\sigma=(23)=\pi\sigma\pi$.
\end{proof}

\subsection{Nichols algebras of braided racks}
We focus on the Nichols algebra $\B(X,q)$ of a braided rack $X$ with respect to a $2$-cocycle $q$, over an arbitrary base field $K$.

We start with the quadratic approximation of $\B(X,q)$, i.e.~the associative algebra covering it defined by quadratic relations solely. We use the formula from \cite{zbMATH05942621}, where it is stated and applied only for fields of characteristic zero. However, the formula also holds for arbitrary characteristics. For further information on the space of quadratic relations over arbitrary fields, we refer to \cite{zbMATH05918254}.


\begin{lem}[{\cite{zbMATH05942621,zbMATH05918254}}] 
\label{lem:S2}
    Let $(X,\triangleright)$ be a finite rack and let $q$ be a $2$-cocycle. Suppose that for each Hurwitz orbit $C\subseteq X^2$, we have $\prod_{(x,y)\in C} q_{xy}=(-1)^{|C|}$. Then there exists a linear combination of all $2$-fold tensors corresponding to the elements in $C$ with non-zero coefficients, lying in the kernel of the quantum symmetrizer $\mathcal{S}_2=I+c_1$.
    
    Moreover, if $q\equiv -1$ is a constant cocycle, then we can take the coefficients in that linear combination to be equal to one.
\end{lem}

\begin{proof}
Let 
\[
C=\{(x_1,x_0),\dots,(x_m,x_{m-1})\},
\]
where $\sigma_1(x_{i},x_{i-1})=(x_{i+1},x_{i})$ for $1\leq i\leq m-1$ and $\sigma_1(x_m,x_{m-1})=(x_1,x_0)$. Let $\alpha_1=1$ and $\alpha_{i+1}=-\alpha_i q_{x_{i} x_{i-1}}$ for $1\leq i\leq m-1$. Thus \[
\alpha_m=(-1)^{m-1}q_{x_1 x_0}\cdots q_{x_{m-1} x_{m-2}}.
\]
Let $\xi=\sum_{i=0}^{m-1} \alpha_{i+1}(x_{i+1}\otimes x_i)$. Then
\begin{align*}
c(\xi) & = \sum_{i=0}^{m-1} \alpha_{i+1} c(x_{i+1}\otimes x_i) \\ & = \sum_{i=1}^{m-1} \alpha_{i} q_{x_{i} x_{i-1}} (x_{i+1}\otimes x_{i}) + \alpha_{m}q_{x_m x_{m-1}}(x_1\otimes x_0) \\ & = \sum_{i=1}^{m-1} -\alpha_{i+1} (x_{i+1}\otimes x_{i}) + (-1)^{m-1}q_{x_1 x_0}\cdots q_{x_{m-1} x_{m-2}}q_{x_m x_{m-1}}(x_1\otimes x_0) \\ & = \sum_{i=1}^{m-1} -\alpha_{i+1} (x_{i+1}\otimes x_{i}) - x_1 \otimes x_0 = -\xi
\end{align*}
as required. Notice that, if $q\equiv -1$ is a constant cocycle, then $\xi=\sum_{(x,y)\in C} x\otimes y$ as $\alpha_1=1$ and by induction $\alpha_{i+1}=-\alpha_iq_{x_i x_{i-1}}=-(-1)=1$.
\end{proof}

\begin{prop} 
\label{prop:braided}
Let $X$ be a finite rack and let $q$ be a $2$-cocycle satisfying $q_{xx}=-1$ for all $x\in X$. Suppose that for every $\emptyset \neq S\subseteq X$ and $t\in S^c = X\setminus S$ there exists some $s\in S$ such that, denoting the Hurwitz orbit of $(t,s)$ by $C$,
\begin{enumerate}
    \item $C \cap \left(S^c\times S\right)=\{(t,s)\}$
    \item $\prod_{(x,y)\in C} q_{xy} = (-1)^{|C|}$.
\end{enumerate}
Then $\Gamma_2(\B(X,q))=\emptyset$.
\end{prop}

\begin{proof}
Write $\B=\B(X,q)=\kk\left<X\right>/I$ and notice that $x^2\in I$ for all $x\in X$, since $$(I+c_1)(x\otimes x)=(1 + q_{x x}) (x\otimes x) = 0.$$
Suppose that $M=Ke_0\oplus Ke_1\oplus Ke_2$ is a truncated point module of degree two over $\B(X,q)$ and let $S=\{x\in X|x\cdot e_0\neq 0\}$. Since $M$ is generated in degree zero, we have that $S\neq \emptyset$. Notice that for every $s\in S$, we have that $s\cdot e_0=\lambda_s e_1$ for some $\lambda_s\neq 0$, so $s\cdot e_1=\lambda_s^{-1} s^2\cdot e_0=0$. Hence, again since $M$ is generated in degree zero, we have that $S\subsetneq X$ (for otherwise, $e_2\notin \B \cdot e_0$).
Fix $t\in X$ such that $t\cdot e_1\neq 0$ (so in particular, $t\notin S$). By assumption (1), there exists some $s\in S$ such that, writing the Hurwitz orbit of $(t,s)$ as
$$C = \{(t,s),(u_1,t),\dots,(u_{m-2},u_{m-3}),(s,u_{m-2})\}$$
we have that, except for $(t,s)$, all of the above pairs do not belong to $S^c \times S$. This means that $u_{i+1}u_i\cdot e_0=0$ for all $0\leq i\leq m-2$ (where $u_0=t,u_{m-1}=s$). Indeed, if $(u_{i+1},u_i)\notin S^c\times S$ then either $u_i\notin S$, so $u_i\cdot e_0=0$, or $u_i\in S$ but also $u_{i+1}\in S$, so $u_{i+1}e_1=0$ and so $u_{i+1}u_i\cdot e_0=\lambda_{u_i} u_{i+1} \cdot e_1=0$.
By assumption (2) and by Lemma \ref{lem:S2}, there exist coefficients $\alpha_0,\dots,\alpha_{m-1}\in K^\times$ such that:
$$ \alpha_0 ts + \alpha_1 u_1t + \cdots \alpha_{m-1} su_{m-2} = 0 $$
so, acting on $e_0$, we see that $\alpha_0 ts\cdot e_0=0$, contradicting that $s\cdot e_0\neq 0,\ t\cdot e_1\neq 0$.
\end{proof}

\begin{cor} \label{cor:braided}
Let $(X,\triangleright)$ be a braided rack and let $q\equiv -1$, a constant $2$-cocycle. Then $\Gamma_2(\B(X,q))=\emptyset$.
\end{cor}
\begin{proof}
Let us verify the conditions of Proposition \ref{prop:braided}. Notice that condition (2) is immediate since $q\equiv -1$. Let $S\subseteq X,t\in S^c$ be given. Take $s\in S$ arbitrary.
Let $C$ denote the Hurwitz orbit of $(t,s)$. By Proposition \ref{prop:3}, $C$ either takes the form $\{(t,s),(s,t)\}$ or $\{(t,s),(u,t),(s,u)\}$. Since $s\in S,t\in S^c$, it follows that $(s,t),(u,t),(s,u)\notin S^c\times S$, so in either case, $C\cap (S^c \times S)=\{(t,s)\}$.
\end{proof}


The racks $\Aff(5,2)$ and $\Aff(5,3)$ are not braided, but 
they still satisfy the first condition of Proposition \ref{prop:braided}, namely, that for every $S\neq \emptyset$ and $t\notin S$ there exists $s\in S$ such that the Hurwitz orbit of $(t,s)$, denoted below $C$, intersects $S^c\times S$ only at $(t,s)$ itself.


\begin{exa}
\label{ex:5 2}
Let $X=\Aff(5,2)$. A direct calculation shows that 
the off-diagonal part of $X^2$ partitions to orbits as follows:
\begin{align*}
    \{(0,1),(2,0),(3,2),(1,3)\}, && 
    \{(0,2),(4,0),(1,4),(2,1)\}, \\
    \{(0,3),(1,0),(4,1),(3,4)\}, &&
    \{(0,4),(3,0),(2,3),(4,2)\}, \\
    \{(1,2),(3,1),(4,3),(2,4)\}.
\end{align*}
Let $S\subseteq \Aff(5,2),t\in S^c$. To verify the desired condition, notice that if $|S|=1$ (respectively, $|S|=4$) then the condition holds automatically for any $s\in S$, as then the right (resp., left) coordinate of $S^c\times S$ is fixed, and the projection of every orbit onto either its left or right coordinates is a bijection. If $|S|=2$ then by $2$-transitivity we may assume that $S=\{0,1\}$. For $t=2$, take $s=0$; we have $C\cap (S^c\times S)=\{(2,0)\}$. For $t=3$, take $s=0$; we have $C\cap (S^c\times S)=\{(3,0)\}$. For $t=4$, take $s=1$; we have $C\cap (S^c\times S)=\{(4,1)\}$. 
If $|S|=3$, again by $2$-transitivity, we may assume that $S^c=\{0,1\}$ and that $t=0$. Then for $s=3$ we get $C\cap \left(S^c \times S\right)=\{(0,3)\}$.

\end{exa}

\begin{exa}
\label{ex:5 3}
Let $X=\Aff(5,3)$. 
The off-diagonal part of $X^2$ partitions to orbits as follows:
\begin{align*}
    \{(0,1),(3,0),(4,3),(1,4)\}, &&
    \{(0,2),(1,0),(3,1),(2,3)\}, \\
    \{(0,3),(4,0),(2,4),(3,2)\}, &&
    \{(0,4),(2,0),(1,2),(4,1)\}, \\
    \{(1,3),(2,1),(4,2),(3,4)\}.
\end{align*}
We argue as in Example~\ref{ex:5 2}. Suppose that $S=\{0,1\}$. If $t=2$, we can take $s=1$; if $t=3$, take $s=0$; for $t=4$, take $s=0$. If $S^c=\{0,1\}$ and $t=0$ then $s=2$ works.
\end{exa}

Finally, we notice that the affine rack 
$\Aff(5,4)$ does not satisfy Condition (1) of Proposition~\ref{prop:braided}.

\begin{exa}
Let $X=\Aff(5,4)$. Consider the set $S=\{0,1\}\subseteq X$ and $t=3$. The orbit of $(3,0)$ is 
$$\{(3,0),(1,3),(4,1),(2,4),(0,2)\},$$ which intersects $S^c \times S$ at $\{(3,0),(4,1)\}$; and the orbit of $(3,1)$ is $$\{(3,1),(0,3),(2,0),(4,2),(1,4)\},$$ which intersects $S^c\times S$ at $\{(3,1),(2,0)\}$.
\end{exa}


\section{Finite-dimensional Nichols algebras}

We recall the following parametrization of the space of point modules of a given graded algebra. Let $A=K\left<x_1,\dots,x_d\right>/\left<f_1,f_2,\dots\right>$ be a graded algebra. Given an $n$-truncated point module $M$ over $A$, fix a homogeneous basis $M=K e_0\oplus \cdots \oplus K e_n$ and encode the action of $A$ on $M$ by parameters $\lambda_{i,j}\in K,\ 1\leq i\leq d,\ 0\leq j\leq n-1$ such that:
\[
x_i \cdot e_j = \lambda_{i,j} e_{j+1}.
\]
Thus $([\lambda_{1,0}\colon \cdots \colon \lambda_{d,0}],\dots,[\lambda_{1,n-1}\colon \cdots \colon \lambda_{d,n-1}])\in \mathbb{P}^{d-1}\times \cdots \times \mathbb{P}^{d-1}$. The points in this product of projective spaces corresponding to truncated point modules can be found as the zero locus of the following multilinear system of equations: each relation $f=\sum_I c_Ix_{i_1}\cdots x_{i_t}$ induces the multilinear equations $\sum_I c_I \lambda_{i_1,s+t-1}\cdots \lambda_{i_t,s}=0$ for each $0\leq s\leq n-t$. The set of points in $\left(\mathbb{P}^{d-1}\right)^{\times n}$ solving these multilinearized polynomial equations is in bijection with $\Gamma_n(A)$, the set of $n$-truncated point modules over $A$. 
In particular, the $n$-truncated point modules over $A$ coincide with the $n$-truncated point modules over the cover of $A$ given by relations of degree at most $n$.
See \cite[Proposition 3.9]{ATV} and \cite{Rogalski} for more details.

\begin{prop} \label{prop:cubic roots}
Let $K$ be a field of characteristic $2$, containing $\mathbb{F}_4$. Let $X=\Aff(3,-1)\cong (12)^{\mathbb{S}_3}$, $q\equiv \zeta\in K$ a primitive cubic root of unity. Then $\Gamma_2(\B(X,q))$ is a rational surface and $\Gamma_3(\B(X,q))=\emptyset$.    
\end{prop}


\begin{proof}
Label the generators of $\B(X,q)$ as $a,b,c$. We computed, using Magma~\cite{zbMATH01077111}, that the quadratic defining relations of $\B(X,q)$ are
$$
ca + \zeta bc + \zeta^2ab,\  
 cb + \zeta^2ba + \zeta ac
$$
and there are further higher-degree relations, including $a^3=b^3=c^3=0$ in degree three.
Recall that $\Gamma_2(\B(X,q))$ is determined by the quadratic approximation of $\B(X,q)$:
$$K\left<a,b,c\right>/\left<ca + \zeta bc + \zeta^2ab,\  
 cb + \zeta^2ba + \zeta ac \right>$$
We can thus identify $\Gamma_2(\B(X,q))$ with
$$\left\{([\lambda_a\colon \lambda_b\colon \lambda_c],[\mu_a\colon \mu_b\colon \mu_c])\in \PP^2\times \PP^2\ \Big|\ \substack{\mu_c \lambda_a + \zeta \mu_b\lambda_c + \zeta^2 \mu_a\lambda_b = \\ 
 \mu_c\lambda_b + \zeta^2\mu_b\lambda_a + \zeta \mu_a\lambda_c = 0}\right\}.$$
 The fiber of the projection $\pr_2\colon \Gamma_2(\B(X,q))\rightarrow \PP^2$ is a singleton whenever the vectors $(\mu_c\ \zeta^2\mu_a\ \zeta \mu_b),(\zeta^2\mu_b\ \mu_c\ \zeta\mu_a)$ are linearly independent, since then $[\lambda_a\colon \lambda_b \colon \lambda_c]$ is given as the ($1$-dimensional) nullspace of the matrix
 $$
 \left(
\begin{matrix}
\mu_c & \zeta^2\mu_a & \zeta \mu_b \\
\zeta^2\mu_b & \mu_c & \zeta\mu_a
\end{matrix}
 \right).
 $$
Namely, the fiber of $\pr_2$ is a singleton away from the subset
 $$ \big\{ [\mu_a\colon \mu_b\colon \mu_c]\in \PP^2\ |\  \mu_a^2=\zeta \mu_b\mu_c,\ \mu_b^2=\zeta \mu_a\mu_c,\ \mu_c^2=\zeta \mu_a\mu_b \big\}. $$
 It can be easily seen that for any point in this set, $\mu_a,\mu_b,\mu_c\neq 0$ and hence, by the first equation, such a point must take the form $[\alpha\colon 1\colon \zeta^2 \alpha^2]$. By the second equation $\alpha^3=1$, and the third equation follows from the first two. Conversely, any point $[\alpha\colon 1\colon \zeta^2 \alpha^2],\ \alpha^3=1$ is in the above set. Hence $\pr_2^{-1}(p)$ is a singleton except for $p=[1\colon 1\colon \zeta^2],[\zeta\colon 1\colon \zeta],[\zeta^2\colon 1\colon 1]$, for each of which $\pr_2^{-1}$ is a copy of $\PP^1$. Hence $\pr_2$ restricted to $\Gamma_2(\B(X,q))$ minus a union of three lines induces an isomorphism onto $\PP^2$ minus three points, and $\Gamma_2(\B(X,q))$ is birational to $\mathbb{P}^2$.


We now move to $\Gamma_3(\B(X,q))$, which is contained in the set of points
$$ ([\lambda_a\colon \lambda_b\colon \lambda_c],
     [\mu_a\colon \mu_b\colon \mu_c],
     [\nu_a\colon \nu_b\colon \nu_c]) \in \PP^2\times \PP^2\times \PP^2
$$
satisfying
\begin{gather}
    \mu_c \lambda_a + \zeta \mu_b\lambda_c + \zeta^2 \mu_a\lambda_b = 0 \label{1} \\ 
    \mu_c\lambda_b + \zeta^2\mu_b\lambda_a + \zeta \mu_a\lambda_c = 0 \label{2} \\
    \nu_c \mu_a + \zeta \nu_b\mu_c + \zeta^2 \nu_a\mu_b = 0 \label{3} \\
    \nu_c\mu_b + \zeta^2\nu_b\mu_a + \zeta \nu_a\mu_c = 0 \label{4} \\ 
    \nu_a \mu_a \lambda_a = \nu_b \mu_b \lambda_b = \nu_c \mu_c \lambda_c = 0 \label{5}
\end{gather}
We arrange the coordinates of any such point in a table
\begin{center}
    \begin{tabular}{|l|l|l|}
    \hline
    $\lambda_a$ & $\lambda_b$ & $\lambda_c$ \\
    \hline
    $\mu_a$ & $\mu_b$ & $\mu_c$ \\
    \hline
    $\nu_a$ & $\nu_b$ & $\nu_c$ \\
    \hline
    \end{tabular}
\end{center}
By \eqref{5}, every column must have at least one zero entry. Every row has at least one non-zero entry (representing a point in $\PP^2$). 
Denote the number of non-zero coordinates of $\lambda,\mu,\nu$ by $\wt(\lambda),\wt(\mu),\wt(\nu)$, respectively.

Assume that $\wt(\mu) = 1$. Suppose $\mu_a \neq 0$ (respectively, $\mu_b \neq 0$ or $\mu_c \neq 0$). By constraints \eqref{1} and \eqref{2}, $\lambda_b = \lambda_c = 0$ (resp. $\lambda_a = \lambda_c = 0$ or $\lambda_a = \lambda_b = 0$) and by constraints \eqref{3} and \eqref{4}, $\nu_b = \nu_c = 0$ (resp. $\nu_a = \nu_c = 0$ or $\nu_a = \nu_b = 0$). But then both $\lambda_a$ and $\nu_a$ are non-zero (resp. $\lambda_b,\nu_b\neq 0$ or $\lambda_c,\nu_c\neq 0$), leading to a column of non-zero entries in the table, a contradiction. Hence $\wt(\mu)\geq 2$.

If $\wt(\mu) = 2$ then one of the pairs $(\mu_a, \mu_b)$, $(\mu_a, \mu_c)$, and $(\mu_b, \mu_c)$ has both components non-zero; suppose this is $(\mu_a, \mu_b)$ (the other cases are analogous). Constraints \eqref{1},\eqref{2} yield
\begin{gather}
    \zeta \mu_b\lambda_c + \zeta^2 \mu_a\lambda_b = 0 \\ 
    \zeta^2\mu_b\lambda_a + \zeta \mu_a\lambda_c = 0.
\end{gather}
If $\wt(\lambda) = 1$ then regardless of which one of $\lambda_a, \lambda_b, \lambda_c$ is non-zero, $\zeta^2\mu_b\lambda_a = 0$, $\zeta^2 \mu_a\lambda_b = 0$, or $\zeta \mu_a\lambda_c = 0$, respectively, so $\lambda_a=\lambda_b=\lambda_c=0$, a contradiction. Hence $\wt(\lambda) \geq 2$. If one of $\lambda_a, \lambda_b, \lambda_c$ is zero then $\zeta \mu_a\lambda_c = 0$, $\zeta \mu_b\lambda_c = 0$, or $\zeta^2 \mu_a\lambda_b = 0$, respectively, so at least one additional coordinate of $\lambda$ is zero, contradicting that $\wt(\lambda)\geq 2$. Thus, in fact, $\wt(\lambda) = 3$. 

Applying analogous arguments `forward' instead of `backwards,' namely, using \eqref{3},\eqref{4} instead of \eqref{1},\eqref{2}, we see that $\wt(\nu) = 3$ as well. Now there are $3 + 2 + 3$ non-zero values in the table, so by the pigeonhole principle, at least one of the columns must contain three non-zero entries, a contradiction. 

If $\wt(\mu) = 3$, we argue that none of $\wt(\lambda)$ and $\wt(\nu)$ can be equal to $1$. If two coordinates in $\lambda$ are zero, equation \eqref{1} forces the third one to be zero as well (since all $\mu_a,\mu_b,\mu_c\neq 0$); and if two coordinates in $\nu$ are zero, equation \eqref{3} forces the third one to be zero as well. 
Hence $\wt(\lambda),\wt(\nu)\geq 2$, so $\wt(\lambda)+\wt(\mu)+\wt(\nu)\geq 2+3+2=7$ and again by the pigeonhole principle, at least one of the three columns in the table has no zero entry, a contradiction. 

Since $\wt(\mu)$ cannot be 1, 2, or 3, it follows that $\Gamma_3(\B(X,q))=\emptyset$.
\end{proof}

The next proposition considers the rack $\mathcal{T}=(123)^{\mathbb{A}_4}$. The conjugation action $\mathbb{A}_4 \curvearrowright \Tau$ is faithful and isomorphic to the standard action of $\mathbb{A}_4$ on four items. 
There is also an affine description of this rack: $\Tau \cong \Aff(\mathbb{F}_4,\omega)$ where $\omega\in \mathbb{F}_4$ is a primitive element.


\begin{prop} \label{prop:5184}
Let $K$ be an arbitrary field containing a primitive cubic root of unity $\zeta$. Let $\Tau$ be equipped with a cocycle $(q_{xy})_{x,y}$ taking values $\pm \zeta$, such that the corresponding Nichols algebra $\B(\Tau,q)$ is $5184$-dimensional. Then $\Gamma_2(\B(\Tau,q))\neq \emptyset$ but $\Gamma_3(\B(\Tau,q))=\emptyset$.
\end{prop}

\begin{proof}
Label the generators of $\B(\Tau,q)$ as $a_0,a_1,a_\omega,a_{\omega^2}$, in accordance with the realization of $\Tau$ as $\Aff(\mathbb{F}_4,\omega)$. We computed, using Magma~\cite{zbMATH01077111}, that the quadratic defining relations of  $\B(\Tau,q)$ are 
\begin{align}
& -\zeta^2 a_0 a_1 -\zeta a_1 a_\omega + a_\omega a_0, \label{4.8} \\
& -\zeta^2 a_0 a_\omega -\zeta a_\omega a_{\omega^2} + a_{\omega^2} a_0, \label{4.9} \\ 
& \zeta a_0 a_{\omega^2} -\zeta^2 a_1 a_0 + a_{\omega^2} a_1, \label{4.10} \\
& \zeta a_1 a_{\omega^2} + \zeta^2 a_\omega a_1 + a_{\omega^2} a_\omega, \label{4.11}
\end{align}
and all generators are cubed zero (degree three relations).

View $\Gamma_2(\B(\Tau,q))\subseteq \mathbb{P}^3 \times \mathbb{P}^3$ with coordinates $$([\lambda_0\colon\lambda_1\colon \lambda_{\omega}\colon \lambda_{\omega^2}],[\mu_0\colon\mu_1\colon \mu_{\omega}\colon \mu_{\omega^2}]),$$ defined by the system of equations:
$$
\underbrace{\left(\begin{matrix}
-\zeta^2\mu_1 & \mu_{\omega^2} & 0 & \zeta \mu_0 \\ 0 & \zeta^2 \mu_{\omega} & \mu_{\omega^2} & \zeta \mu_1 \\ -\mu_{\omega} & \zeta^2 \mu_0 & \zeta \mu_1 & 0 \\ -\mu_{\omega^2} & 0 & \zeta^2 \mu_0 & \zeta \mu_{\omega}
\end{matrix}\right)}_{=: M} \left(\begin{matrix}
    \lambda_0 \\ \lambda_1 \\ \lambda_{\omega} \\ \lambda_{\omega^2}
\end{matrix}\right)=\vec{0}$$
(following from relations \eqref{4.10},\eqref{4.11},\eqref{4.8},\eqref{4.9}, respectively.)
Denote by $\pr_1,\pr_2$ the natural projections $\Gamma_2(\B(\Tau,q)) \subseteq \mathbb{P}^3 \times \mathbb{P}^3 \rightarrow \PP^3$. Observe that $$\pr_2(\Gamma_2(\B(\Tau,q))) = \{[\mu_0\colon\mu_1\colon \mu_{\omega}\colon \mu_{\omega^2}]\in \PP^3\ |\ \det(M)=0\}=:\mathcal{Z} \neq \emptyset$$
so we have a surjective map $\pr_2\colon \Gamma_2(\B(\Tau,q)) \twoheadrightarrow \mathcal{Z}$.

Let $\Gamma_3(\B(\Tau,q))$ have coordinates
$$([\lambda_0\colon\lambda_1\colon \lambda_{\omega}\colon \lambda_{\omega^2}],[\mu_0\colon\mu_1\colon \mu_{\omega}\colon \mu_{\omega^2}],[\nu_0\colon\nu_1\colon \nu_{\omega}\colon \nu_{\omega^2}]).$$
As before, we have the following matrix equations:
\begin{align*}   
\underbrace{\left(\begin{matrix}
-\zeta^2\mu_1 & \mu_{\omega^2} & 0 & \zeta \mu_0 \\ 0 & \zeta^2 \mu_{\omega} & \mu_{\omega^2} & \zeta \mu_1 \\ -\mu_{\omega} & \zeta^2 \mu_0 & \zeta \mu_1 & 0 \\ -\mu_{\omega^2} & 0 & \zeta^2 \mu_0 & \zeta \mu_{\omega}
\end{matrix}\right)}_{=:M} \underbrace{\left(\begin{matrix}
    \lambda_0 \\ \lambda_1 \\ \lambda_{\omega} \\ \lambda_{\omega^2}
\end{matrix}\right)}_{=:\vec{\lambda}} &= \underbrace{\left(\begin{matrix}
-\zeta^2\nu_1 & \nu_{\omega^2} & 0 & \zeta \nu_0 \\ 0 & \zeta^2 \nu_{\omega} & \nu_{\omega^2} & \zeta \nu_1 \\ -\nu_{\omega} & \zeta^2 \nu_0 & \zeta \nu_1 & 0 \\ -\nu_{\omega^2} & 0 & \zeta^2 \nu_0 & \zeta \nu_{\omega}
\end{matrix}\right)}_{=:N}
\underbrace{\left(\begin{matrix}
    \mu_0 \\ \mu_1 \\ \mu_{\omega} \\ \mu_{\omega^2}
\end{matrix}\right)}_{=:\vec{\mu}} \\ & = \vec{0}.
\end{align*}
We denote $\lambda=[\lambda_0\colon\lambda_1\colon \lambda_{\omega}\colon \lambda_{\omega^2}]$ etc.; recall that $\wt$ denotes the number of non-zero coordinates in a given projective point. For the rest of the proof, it will be convenient to switch indexing and enumerate $\lambda_2 = \lambda_\omega$ and $\lambda_3 = \lambda_{\omega^2}$ (likewise for $\mu,\nu$). We denote the $i$-th column of a matrix $A$ by $C_i(A)$ and index its columns $0,1,\dots$.

Suppose that $\wt(\mu) = 1$ with $\mu_i \neq 0$ ($0\leq i\leq 3$). It follows that $\vec{0} = N \vec{\mu} = \mu_i C_i(N)$ and so $C_i(N) = \vec{0}$. Observe that no matter what $i$ is, the column $C_i(N)$ contains entries that are equal (up to non-zero coefficients) to $\nu_j$ for all $j \neq i$; thus $\nu_j = 0$ for all $j \neq i$ and $\nu_i$ must be non-zero. 
Now, since $C_i(M)$ does not involve $\mu_i$, $C_i(M) = \vec{0}$. Notice that $\mu_i$ appears in a different row in each one of the columns $C_j(M),\ j\neq i$, and since the entries involving $\mu_i$ are the only non-zero ones in $M$, it follows that $C_j(M),\ j\neq i$ are linearly independent. Since $M\vec{\lambda}=\vec{0}$, we have that $\sum_{j\neq i} \lambda_j C_j(M)=\vec{0}$ and it follows that $\lambda_j = 0$ for all $j \neq i$, and therefore necessarily $\lambda_i \neq 0$. We proved that $\lambda_i,\mu_i,\nu_i$ are all non-zero. But since all generators of $\B(\Tau,q)$ are cubed zero, we have $\lambda_i\mu_i\nu_i=0$ for all $0\leq i\leq 3$, a contradiction. It follows that $\wt(\mu)\geq 2$.

We further claim that $\wt(\mu)\geq 3$; suppose that $\wt(\mu) = 2$. Letting $*$ denote an arbitrary non-zero scalar, below are the six possible forms that $M$ can take, depending on which two coordinates of $\mu$ are non-zero:
\begin{gather*}
    \begin{pmatrix}
        * & 0 & 0 & * \\ 0 & 0 & 0 & * \\ 0 & * & * & 0 \\ 0 & 0 & * & 0
    \end{pmatrix},
    \begin{pmatrix}
        0 & 0 & 0 & * \\ 0 & * & 0 & 0 \\ * & * & 0 & 0 \\ 0 & 0 & * & *
    \end{pmatrix},
    \begin{pmatrix}
        0 & * & 0 & * \\ 0 & 0 & * & 0 \\ 0 & * & 0 & 0 \\ * & 0 & * & 0
    \end{pmatrix}, \\
    \begin{pmatrix}
        * & 0 & 0 & 0 \\ 0 & * & 0 & * \\ * & 0 & * & 0 \\ 0 & 0 & 0 & *
    \end{pmatrix},
    \begin{pmatrix}
        * & * & 0 & 0 \\ 0 & 0 & * & * \\ 0 & 0 & * & 0 \\ * & 0 & 0 & 0
    \end{pmatrix},
    \begin{pmatrix}
        0 & * & 0 & 0 \\ 0 & * & * & 0 \\ * & 0 & 0 & 0 \\ * & 0 & 0 & *
    \end{pmatrix},
\end{gather*}
The matrices are arranged in the following order of non-zero components: $(0, 1)$, $(0, 2)$, $(0, 3)$, $(1, 2)$, $(1, 3)$, $(2, 3)$. It is straightforward to verify that all of the above forms represent invertible matrices.
Since $M \vec{\lambda} = \vec{0}$, it follows that $\vec{\lambda}=\vec{0}$, which is impossible as $\lambda$ represents a point in $\mathbb{P}^3$. Thus $\wt(\mu) \neq 2$.

Now suppose that $\wt(\nu) = 1$ with $\nu_i \neq 0$ ($0\leq i\leq 3$). Since $\nu_i$ appears in all columns $C_j(N),\ j\neq i$ and furthermore it appears in different rows in each of these columns, it follows that these column vectors are linearly independent. Moreover, since $C_i(N)$ involves only $\nu_j,\ j\neq i$, which all vanish, $C_i(N)=\vec{0}$. 
Hence $\vec{0} = N\vec{\mu} = \sum_{j\neq i} \mu_j C_j(N)$, and it follows that  $\mu_j = 0$ for all $j \neq i$ and therefore necessarily $\mu_i \neq 0$. This contradicts that $\wt(\mu) \geq 2$ as established above. It follows that $\wt(\nu)\geq 2$. The argument given above for why $\wt(\mu)\neq 2$ shows mutatis mutandis that $\wt(\nu)\neq 2$ (argued for the matrix $N$ instead of $M$), so $\wt(\nu)\geq 3$ as well.

We now consider $\lambda$. First, if $\wt(\lambda)=1$, say, $\lambda_i\neq 0$ ($0\leq i\leq 3$) and $\lambda_j=0$ for all $j\neq i$, then $\vec{0} = M \vec{\lambda} = \lambda_i C_i(M)$ so $C_i(M)=\vec{0}$. Since $C_i(M)$ contains entries that are equal (up to non-zero coefficients) to $\mu_j$ for all $j \neq i$, it follows that $\mu_j = 0$ for all $j \neq i$ and $\wt(\mu)=1$, contradicting the above. Hence $\wt(\lambda)\geq 2$. 

Suppose that $\wt(\lambda) = 2$. For any two indices $0\leq i < j\leq 3$, the $4\times 2$ submatrix of $M$ consisting of $C_i(M),C_j(M)$ has one row (say, $p$) where the first entry is zero and another row (say, $q$) where the second entry is zero. Notice that the other entries in these rows are, up to non-zero coefficients from $K$,  $\mu_k, \mu_l$ where $\{i, j, k , l\}=\{0,1,2,3\}$. 
Now suppose that $\lambda_i$ and $\lambda_j$ are the non-zero coordinates of $\lambda$. Since $\vec{0}=M\vec{\lambda}=\lambda_i C_i(M)+\lambda_jC_j(M)$, considering the $p$ and $q$ entries in this linear combination, it follows that $\mu_k$ and $\mu_l$ must both be zero. Therefore $\wt(\mu)\leq 2$, contradicting the above arguments. Hence $\wt(\lambda)\geq 3$ as well.

Finally, we showed that $\wt(\lambda),\wt(\mu), \wt(\nu) \geq 3$. It follows that within the table  
\begin{center}
    \begin{tabular}{|l|l|l|l|}
    \hline
    $\lambda_0$ & $\lambda_1$ & $\lambda_2$ & $\lambda_3$ \\
    \hline
    $\mu_0$ & $\mu_1$ & $\mu_2$ & $\mu_3$ \\
    \hline
    $\nu_0$ & $\nu_1$ & $\nu_2$ & $\nu_3$ \\
    \hline
    \end{tabular}
\end{center}
there are at most $3$ zero entries. However, since $\B(\Tau,q)$ satisfies the relations $a_0^3=a_1^3=a_\omega^3=a_{\omega^2}^3=0$, it follows that $\lambda_0\mu_0\nu_0=\cdots=\lambda_{\omega^2}\mu_{\omega^2}\nu_{\omega^2}=0$, so each one of the columns in this table must have at least one zero entry, a contradiction. Therefore, $\Gamma_3(\B(\Tau,q)) = \emptyset$.
\end{proof}

We are now ready to prove Theorem \ref{thm:table}.

\begin{proof}[{Proof of Theorem \ref{thm:table}}]
We go over the racks and cocycles listed in Table \ref{tab:examples} case by case.

\medskip

\noindent \emph{The rack $(12)^{\mathbb{S}_3}$ equipped with the constant cocycle $-1$} yields a Nichols algebra isomorphic to the Fomin--Kirillov algebra $\E_3$, whose $\Gamma_2$ is empty by \cite{GMSW}. Equipped with the constant cocycle $\zeta_3$ over a field of characteristic $2$, the rack $(12)^{\mathbb{S}_3}$ gives a Nichols algebra for which $\Gamma_2\neq \emptyset$ but $\Gamma_3=\emptyset$ as shown in Proposition \ref{prop:cubic roots}.

\medskip

\noindent \emph{The rack $\mathcal{T}$} is a braided rack and thus, equipped with the constant cocycle $-1$, gives a Nichols algebra with $\Gamma_2=\emptyset$ by Corollary \ref{cor:braided}. This applies also to the case where the characteristic of the base field is $2$ and $q\equiv -1=1$. Equipped with the cocycle $\psi$ (taking values $\pm \zeta_3$), we get a (very) different Nichols algebra of dimension $5184$, which is shown to have $\Gamma_2\neq \emptyset$ but $\Gamma_3=\emptyset$ in Proposition \ref{prop:5184}.

\medskip

\noindent \emph{The racks $\Aff(5,2),\Aff(5,3)$}, as analyzed in Examples \ref{ex:5 2} and \ref{ex:5 3}, satisfy Condition (1) of Proposition \ref{prop:braided}. Hence their Nichols algebras with the constant cocycle $-1$ have $\Gamma_2 = \emptyset$ by that proposition.

\medskip

\noindent \emph{The rack $(12)^{\mathbb{S}_4}$} equipped with the cocycle $\chi$ gives the Fomin--Kirillov algebra $\E_4$ for which $\Gamma_2=\emptyset$ (see \cite{GMSW}). Since this rack is braided, it follows from Corollary \ref{cor:braided} that, when equipped with the constant cocycle $-1$, the resulting Nichols algebra has $\Gamma_2=\emptyset$. 

\medskip

\noindent \emph{The rack $(1234)^{\mathbb{S}_4}$} is braided by Proposition \ref{prop:conj S_n braided} and therefore its Nichols algebra with respect to the constant cocycle $-1$ has $\Gamma_2=\emptyset$ by Corollary \ref{cor:braided}.

\medskip

\noindent \emph{The racks $\Aff(7,3)$ and $\Aff(7,5)$} are both braided, so by Corollary \ref{cor:braided}, the Nichols algebras associated with them with respect to the constant cocycle $-1$ have $\Gamma_2=\emptyset$.

\medskip

\noindent Finally, \emph{the rack $(12)^{S_5}$} is braided; with cocycle $\chi$ it yields a Fomin--Kirillov algebra $\E_5$, whose $\Gamma_2$ is empty by \cite{GMSW}; and with constant cocycle $-1$, we obtain the same result by Corollary \ref{cor:braided}.
\end{proof}

\section{Affine racks at the cocycle -1} \label{sec:affine racks}


Let $K$ be a field of characteristic zero. 
Let $q=p^f$ and let $\alpha\in \mathbb{F}_q^\times$, $\alpha\neq 1$. Recall that $(\Aff(q,\alpha),\triangleright)$ is the rack whose underlying set is $X=\{x_g|g\in \mathbb{F}_q\}$ and $x_g\triangleright x_h=x_{(1-\alpha)g+\alpha h}$. When $q$ is prime, it is known that the second cohomology groups of such racks are trivial~\cite[Lemma 5.1]{zbMATH02100251}, so all $2$-cocycles are constant. 

\begin{lem} \label{lem:lin alg} Over a field of characteristic zero, suppose that
\begin{align}
    \underbrace{\left( \begin{matrix}
        \mu_{\pi(0)} & \mu_{\pi(1)} & \cdots & \mu_{\pi(p-1)} \\
        \mu_{\pi(p-1)} & \mu_{\pi(0)} & \cdots & \mu_{\pi(p-2)} \\
        \vdots & \vdots & \ddots & \vdots \\
        \mu_{\pi(1)} & \mu_{\pi(2)} & \cdots & \mu_{\pi(0)} \\
        \mu_{0} & 0 & \cdots & 0 \\
        0 & \mu_{1} & \cdots & 0 \\
        \vdots & \vdots & \ddots & \vdots \\
        0 & 0 & \cdots & \mu_{p-1}
    \end{matrix} \right)}_{M}
    \underbrace{\left( \begin{matrix}
        u_{0} \\ u_{1} \\ \vdots \\ u_{p-1}
    \end{matrix} \right)}_{\vec{u}} = \left( \begin{matrix}
        0 \\ 0 \\ \vdots \\ 0
    \end{matrix} \right)
\end{align}
for some permutation $\pi$ on $\mathbb{F}_p$. Then either $M=0$ or $\vec{u}=0$.
\end{lem}

\begin{proof}
By extending scalars, we may assume we are working over $\mathbb{C}$. Notice that the first $p$ rows of $M$ form a circulant $p\times p$ matrix $N$, and therefore its eigenvalues and eigenvectors are: 
\begin{align*}
\left\{\nu_k := \sum_{i=0}^{p-1} \mu_{\pi(i)}\zeta_p^{ki}\right\}_{k=0,\dots,p-1},\ \ \ \left\{v_k := \left( \begin{matrix}
        1 & \zeta_p^k & \cdots & \zeta_p^{(p-1)k}
    \end{matrix} \right)^{T}\right\}_{k=0,\dots,p-1}
\end{align*}
where $\zeta = \zeta_p$ is a primitive $p$-th root of unity. 
Notice that $v_0,\dots,v_{p-1}$ are linearly independent (they form a Vandermonde matrix). 
We index the entries of all vectors and matrices starting from $0$. 
Indeed, for any $0\leq k\leq p-1$ and $0\leq l\leq p-1$,
\begin{align*}
    \left(N\cdot v_k\right)_l & = \sum_{r=0}^{p-1} \mu_{\pi(\underbrace{-l+r}_{=: i})} \zeta^{kr} \\
    & = \sum_{i=0}^{p-1} \mu_{\pi(i)} \zeta^{k(i+l)} = \nu_k \cdot \zeta^{kl} = \nu_k (v_k)_l.
\end{align*}
Let $T=\{t_1,\dots,t_k\}\subseteq \{0,1,\dots,p-1\}$ be the set of indices for which $\nu_{t_1}=\cdots=\nu_{t_k}=0$, equivalently, $v_{t_1},\dots,v_{t_k}\in \ker(N)$. If $T=\emptyset$ then $N$ has $p$ linearly independent eigenvectors with non-zero eigenvalues, so $N$ -- and hence $M$ -- are of rank $p$; therefore $\ker(M)=\{\vec{0}\}$ and necessarily $\vec{u}=\vec{0}$, and we are done. We henceforth assume $T\neq \emptyset$. 
Thus:
\begin{align} \label{Vandermonde}
\left( \begin{matrix}
1 & \zeta^{t_1} & \cdots & \zeta^{t_1(p-1)} \\
\vdots & \vdots & \ddots & \vdots \\
1 & \zeta^{t_k} & \cdots & \zeta^{t_k(p-1)}
\end{matrix} \right) \left( \begin{matrix}
    \mu_{\pi(0)} \\ \vdots \\ \mu_{\pi(p-1)}
\end{matrix} \right) = \left( \begin{matrix}
    \nu_{t_1} \\ \vdots \\ \nu_{t_k}
\end{matrix} \right) = \vec{0}.
\end{align}
Furthermore, since $\vec{u}\in \ker(N)$, there is a linear combination $\vec{u} = \alpha_1 v_{t_1}+\cdots+\alpha_kv_{t_k}$. Let $S=\{s_1,\dots,s_l\}\subseteq \{0,1,\dots,p-1\}$ be the set of indices for which $u_{s_i}=0$. Since the $s$-th entries of $v_{t_1},\dots,v_{t_k}$ are, respectively, $(\zeta^s)^{t_1},\dots,(\zeta^s)^{t_k}$, we have that $\alpha_1 (\zeta^{s_i})^{t_1}+\cdots+\alpha_k (\zeta^{s_i})^{t_k}=0$ for all $1\leq i\leq l$. Hence we have $l$ distinct $p$-th roots of unity as zeros of the polynomial $f(X) = \alpha_1 X^{t_1}+\cdots+\alpha_k X^{t_k}$, whose degree is at most $p-1$. By \cite[Observation on pages 125--126]{Tao}, it follows that $f(X)$ is either the zero polynomial (in which case $\alpha_1=\cdots=\alpha_k=0$ and hence $\vec{u}=0$, and we are done), or else, $f(X)$ has at least $l+1$ non-zero monomials; therefore $k\geq l+1$, equivalently, $l\leq k-1$.

Notice that for all $r\notin S$, we have $u_r\neq 0$ and, considering the $(p+r)$-th row of $M$, we obtain that $\mu_r=0$. By the above assertion, there are at least $p-l\geq p-k+1$ such $\mu_r$'s. Hence the number of non-zero monomials in the polynomial
\[
g(X) = \mu_{\pi(0)}+\mu_{\pi(1)}X+\cdots+\mu_{\pi(p-1)}X^{p-1}
\]
is at most $k-1$.
Using \eqref{Vandermonde}, we obtain that the polynomial $g(X)$ 
has $k$ distinct $p$-th roots of unity as zeros, namely, $\zeta^{t_1},\dots,\zeta^{t_k}$. Again by \cite{Tao}, it follows that $g(X)=0$ so $\mu_{\pi(0)}=\cdots=\mu_{\pi(p-1)}=0$ and hence $M=0$, as required. 
\end{proof}

\begin{proof}[{Proof of Theorem \ref{thm:affine}}]
Fix $\alpha\in \mathbb{F}_p^\times$. Let $X=\Aff(\mathbb{F}_p,\alpha)$ and let $\B=\B(X,q)$ be the Nichols algebra associated with $X$ with respect to the constant $2$-cocycle $q\equiv -1$.
The action of $B_2=\left<\sigma\right>\curvearrowright X^2
$  is given by the matrix
$$
c = \left( \begin{matrix}
1-\alpha & \alpha \\ 1 & 0
\end{matrix} \right).
$$
For each $d\in \mathbb{F}_p$, consider the subset $$C_{d} = \{(-\alpha g + d, g)\ |\ g\in \mathbb{F}_p\}\subseteq X^2.$$
We claim that $C_{d}$ is $B_2$-invariant. Indeed,
$$\left( \begin{matrix}
1-\alpha & \alpha \\ 1 & 0
\end{matrix} \right) \left( \begin{matrix}
    -\alpha g + d \\ g
\end{matrix} \right) = \left( \begin{matrix}
    \alpha^2 g + (1-\alpha) d \\ -\alpha g + d
\end{matrix} \right) = \left( \begin{matrix}
    -\alpha(-\alpha g + d) + d \\ -\alpha g + d
\end{matrix} \right).
$$
It follows that $C_{d}$ is a disjoint union of Hurwitz orbits. Since the cocycle is $q_{xy}\equiv -1$, we have by Lemma \ref{lem:S2} that the sum over every orbit in $X^2$ gives a quadratic relation in $\B$. Adding up the relations corresponding to the orbits participating in $C_d$, we obtain the following relation in $\B$:
\begin{align} \label{quadratic relations}
\sum_{g \in \mathbb{F}_p} x_{-\alpha g + d}x_{g} = 0.
\end{align}
Suppose that $M=Ke_0\oplus Ke_1\oplus Ke_2$ is a graded $\B$-module and let $$x_g\cdot e_i=\lambda_{g,i}\cdot e_{i+1},\ \ i=0,1.$$ By \eqref{quadratic relations},
\begin{align} \label{relations lambdas}
\sum_{g\in \mathbb{F}_p} \lambda_{-\alpha g+d,1}\lambda_{g,0}=0.
\end{align} Organizing this system of quadratic equations in a matrix form, we obtain:

\begin{align} \label{matrix}
\underbrace{\left(
\begin{matrix}
     \lambda_{0,1} & \lambda_{-\alpha,1} & \cdots & \lambda_{-\alpha(p-1),1} \\
     \lambda_{1,1} & \lambda_{-\alpha+1,1} & \cdots & \lambda_{-\alpha(p-1)+1,1} \\ \vdots & \vdots & \ddots & \vdots \\
     \lambda_{p-1,1} & \lambda_{-\alpha+(p-1),1} & \cdots & \lambda_{-\alpha(p-1)+p-1,1}
\end{matrix}
\right)}_{=:\Lambda} \left(\begin{matrix}
    \lambda_{0,0} \\ \vdots \\ \lambda_{p-1,0}
\end{matrix}\right) = \vec{0}
\end{align}
where $\Lambda_{i,j}=\lambda_{-\alpha j+i,1}$, indexing the rows and columns of $\Lambda$ by $0,1,\dots,p-1$.
Notice that
$$\Lambda_{i,j-1}=\lambda_{-\alpha (j-1)+i,1}=\lambda_{-\alpha j + i + \alpha,1} = \Lambda_{i + \alpha,j}$$
so $R_{\alpha e}$ is the $e$-th cyclic shift to the right of the first row. Re-ordering the rows of $\Lambda$ (namely, $R_0,R_{\alpha},\dots,R_{(p-1)\alpha}$), we obtain a circulant matrix. Furthermore, for each $g\in \mathbb{F}_p$ we have that $x_g^2=0$ holds in $\B$: indeed, $c_1(x_g\otimes x_g)=-x_g\otimes x_g$, so $x_g\otimes x_g$ is in the kernel of $I+c_1$; therefore, $\lambda_{g,1}\lambda_{g,0}=0$. Appending these equations to \eqref{matrix}:
\begin{align}
    \left( \begin{matrix}
        \lambda_{0,1} & \lambda_{-\alpha,1} & \cdots & \lambda_{-\alpha(p-1),1} \\
        \lambda_{-\alpha(p-1),1} & \lambda_{0,1} & \cdots & \lambda_{-\alpha(p-2),1} \\
        \vdots & \vdots & \ddots & \vdots \\
        \lambda_{-\alpha,1} & \lambda_{-2\alpha,1} & \cdots & \lambda_{0,1}  \\ \lambda_{0,1} & 0 & \cdots & 0 \\
        0 & \lambda_{1,1} & \cdots & 0 \\
        \vdots & \vdots & \ddots & \vdots \\
        0 & 0 & \cdots & \lambda_{p-1,1}
    \end{matrix} \right)
    \left( \begin{matrix}
        \lambda_{0,0} \\ \lambda_{1,0} \\ \vdots \\ \lambda_{p-1,0}
    \end{matrix} \right) = \left( \begin{matrix}
        0 \\ 0 \\ \vdots \\ 0
    \end{matrix} \right).
\end{align}
This fits into the framework of Lemma \ref{lem:lin alg} with $\mu_i=\lambda_{i,1}$ and $\pi(i)=-\alpha i \mod p$; therefore, by Lemma \ref{lem:lin alg}, either $\lambda_{g,0}=0$ for all $g\in \mathbb{F}_p$, or $\lambda_{g,1}=0$ for all $g\in \mathbb{F}_p$. In either case, $M$ cannot be generated by its degree-$0$ part, a contradiction. It follows that $\Gamma_2(\B)=\emptyset$.
\end{proof}


\section{Higher roots of unity} \label{sec:higher roots}

\begin{prop}
Suppose $X$ is a rack and $q_{xy}\equiv q$ is a constant $2$-cocycle over a base field $K$. Then $p(\B(X,q))\geq \ord_{K^\times}(q)-1$.
\end{prop}
\begin{proof}
It suffices to prove the claim in case that $q\neq 1$ is a root of unity.
Let $V=\Span_K X$ and consider the tensor algebra $T(V)=\bigoplus_{n\geq 0} V^{\otimes n}$ (where $V^{\otimes 0} := K$). 
Notice that, since the cocycle is constant, for each pure tensor $u\in X^{\otimes n}$ and $1\leq i \leq n-1$ we have $c_i(u)=qu'$ where $u'$ is another pure tensor. 

Let $f_n\colon V^{\otimes n}\rightarrow K$ be the linear functional mapping all tensors to $1$ and let $I_n=\ker(f_n)$; observe that $T(V)/\bigoplus_{n\geq 1} I_n\cong K[t]$. Let $p_n\colon V^{\otimes n+1}\rightarrow V^{\otimes n}$ be the linear map erasing the first component of any $(n+1)$-fold tensor. Notice that 
\begin{equation} \label{eq:otimes}
f_{n+1}\circ (I\otimes T)=f_n\circ T \circ p_n
\end{equation}
for any $T\colon V^{\otimes n} \rightarrow V^{\otimes n}$. Indeed, if $u=u_1\otimes\cdots\otimes u_{n+1}$ is a pure $(n+1)$-tensor and $T(u_2\otimes\cdots\otimes u_{n+1})=\sum \alpha_i v_i$ where $v_i$ are pure $n$-tensors, then $f_n(T(p_n(u)))=\sum_i \alpha_i$, and also $f_{n+1}\left((I\otimes T)(u)\right)=f_{n+1}(\sum_i \alpha_i u_1\otimes v_i)=\sum_i \alpha_i$.


We claim that 
\begin{equation} \label{eq:fnsn} f_n \circ \mathcal{S}_n=\underbrace{\prod_{i=1}^{n-1}(1+q+\cdots+q^i)}_{=:q(n)} f_n. 
\end{equation}
This can be seen by induction on $n$. It suffices to verify this for pure tensors. Indeed, if $u=u_1\otimes\cdots\otimes u_{n+1}$ is a pure $(n+1)$-fold tensor then
\begin{eqnarray*}
\mathcal{S}_{n+1}(u) & = & (I \otimes \mathcal{S}_n)(I+c_1+c_1c_2+\dots+c_1\cdots c_n)(u) \\
& = & (I\otimes \mathcal{S}_n)(u^{(0)}+qu^{(1)}+\dots+q^nu^{(n)}) \\ & = & \sum_{i=0}^{n} q^i (I\otimes \mathcal{S}_n)(u^{(i)})
\end{eqnarray*}
where $u^{(i)}$ are all pure $(n+1)$-tensors. 
Hence by \eqref{eq:otimes} $$\left(f_{n+1} \circ \mathcal{S}_{n+1}\right)(u)=\sum_{i=0}^{n} q^i f_{n+1}\left(\left(I\otimes \mathcal{S}_n\right)(u^{(i)})\right)=\sum_{i=0}^n q^i \left(f_n \circ \mathcal{S}_n\right)(p_n(u^{(i)}))$$
and by the induction hypothesis, this is equal to
$$\sum_{i=0}^n q^i \prod_{j=1}^{n-1} (1+q+\cdots+q^j) = (1+q+\cdots+q^n) \cdot \prod_{j=1}^{n-1} (1+q+\cdots+q^j) = \prod_{j=1}^{n} (1+q+\cdots+q^j),$$
as required.

Suppose that $n$ is such that $q(n)\neq 0$. By \eqref{eq:fnsn}, for every $\xi\in \ker \mathcal{S}_n$, we have that $q(n)f_n(\xi) = f_n\left(\mathcal{S}_n(\xi)\right) = 0$ and thus $f_n(\xi)=0$. It follows that for all $n$ for which $q(n)\neq 0$, we have $\ker\mathcal{S}_n \subseteq \ker f_n = I_n$.

Notice that for all $n\leq \ord_{K^\times}(q) - 1$, we have $$q(n)=\frac{\prod_{i=1}^{n-1} (q^{i+1} - 1)}{(q-1)^{n-1}} \neq 0$$
and therefore, $\ker \mathcal{S}_n \subseteq I_n$ for all $n\leq \ord_{K^\times}(q) - 1$, so there is a surjection
$\B(X,q)\twoheadrightarrow K[t]/(t^{\ord_{K^\times}(q)})$ given by mapping each $d$-fold pure tensor to $t^d$. Subsequently, $\B(X,q)$ admits a truncated point module of degree $\ord_{K^\times}(q)-1$, by inflation.
\end{proof}

We now focus on the simplest indecomposable case: the rack $X=\Aff(3,-1)=(12)^{\mathbb{S}_3}$. We saw in Theorem \ref{thm:table} that, equipped with the cocycle $q\equiv -1$, the resulting Nichols algebra has $\Gamma_2=\emptyset$, and with $q\equiv \zeta_3$ a primitive cubic root of unity over a field of characteristic $2$, then $\Gamma_2\neq \emptyset$ but $\Gamma_3=\emptyset$. We now demonstrate how the situation becomes much more complicated when $q\equiv \zeta_3$ for a field of characteristic zero.

The following relations of $\B=\B(X,q)$ have been computed using Magma~\cite{zbMATH01077111}. We fix a presentation of $\B$ with generators $a,b,c$.

\medskip

\noindent \emph{Quadratic relations:} None.

\medskip

\noindent \emph{Cubic relations:}
\begin{gather*}
  a^3,\quad 
  b^3,\quad 
  c^3,\quad a^2c + aba + acb + bab + b^2c + bca + ca^2 + cb^2,\\
  ab^2 + ac^2 + bac + b^2a + bcb + cab + cbc + c^2a, \\
 a^2b + abc + aca + ba^2 + bc^2 + cac + cba + c^2b
\end{gather*}

\medskip

\noindent \emph{Quartic relation:}
\begin{multline*}
   (3\zeta_3+4\zeta_3^2)a^2b^2 + 2\zeta_3 a^2c^2 + (3\zeta_3+\zeta_3^2)abac +
2\zeta_3 ab^2a - 3abcb\\
+ (2\zeta_3+3\zeta_3^2)acab
+ (2\zeta_3+\zeta_3^2)acbc 
-2\zeta_3 ba^2b + babc + 3 baca \\
+ (-3\zeta_3-4\zeta_3^2)b^2a^2 -\zeta_3^2 bcac +
(-2\zeta_3-3\zeta_3^2) bcba 
\\
+ 2\zeta_3 ca^2c -\zeta_3^2 caba - cacb + \zeta_3^2 cbab + cbca.
\end{multline*}
 
Notice that the non-monomial cubic relations take the form: \begin{align*} x_1 \cdot \sum \{\text{Hurwitz orbit}\} & + x_2 \cdot \sum \{\text{Hurwitz orbit}\} \\ & + x_3 \sum \{\text{squares, i.e. singleton Hurwitz orbits}\} \end{align*}
for $\{x_1,x_2,x_3\}=\{a,b,c\}$.
 
 \begin{proof}[{Proof of Theorem \ref{prop:zeta3}}]
We have $\Gamma_2=\PP^2 \times \PP^2$ since $\B$ has no quadratic defining relations. We turn to compute $\Gamma_3\left(\B\right) \subseteq \mathbb{P}^2\times \mathbb{P}^2\times \mathbb{P}^2$. 
We represent points in $\Gamma_3\left(\B\right)$ as arrays
\begin{gather}
    \begin{tabular}{|l|l|l|}
    \hline
    $a_0$ & $b_0$ & $c_0$ \\ \hline 
    $a_1$ & $b_1$ & $c_1$ \\ \hline
    $a_2$ & $b_2$ & $c_2$ \\ \hline
    \end{tabular}
\end{gather}
with each row projectively defined. As usual, we let the weight of a point in a projective space denote the number of non-zero coordinates in it. The cubic relations 
\[
a^3=b^3=c^3=0
\]
imply $a_2a_1a_0=b_2b_1b_0=c_2c_1c_0=0$, namely, each column of the array contains at least one zero coordinate.

\bigskip

\noindent \emph{Case I: $\wt([a_0\colon b_0\colon c_0])=1$.} Notice that under the permutation $a\mapsto b,b\mapsto c,c\mapsto a$, the three monomial cubic relations are cyclically permuted, and the three non-monomial cubic relations are cyclically permuted (in fact, the action $\mathbb{S}_3\curvearrowright X$ extends to an action on the cubic relations). We may thus first focus on the case that $a_0=1,b_0=c_0=0$ and then duplicate the resulting points through this action. 

\medskip

\emph{Sub-case 1: $a_1=0$.}
Then $b_1\neq 0$, for otherwise $[a_1\colon b_1\colon c_1]=[0\colon 0\colon 1]$ and the first non-monomial cubic relation implies that $b_2=0$, the second non-monomial cubic relation implies that $c_2=0$, and the third non-monomial cubic relation implies that $a_2=0$; therefore $a_2=b_2=c_2=0$, a contradiction. 

Thus assume $[a_1\colon b_1\colon c_1]=[0\colon 1\colon c_1]$. The non-monomial cubic relations become $$a_2+b_2c_1=b_2+c_2c_1=a_2c_1+c_2=0.$$ It follows that $a_2\neq 0$ (otherwise $a_2=b_2=c_2=0$), so we may assume $a_2=1$. Thus we obtain the equations $1+b_2c_1=b_2+c_2c_1=c_1+c_2=0$ whose solutions are $c_2=1,\zeta_3,\zeta_3^2$ and $c_1=-c_2,b_2=c_2^2$. Permuting the columns (recall that we made an assumption that $a_0=1$), we get the arrays
$$
\begin{tabular}{|l|l|l|}
    \hline
    $1$ & $0$ & $0$ \\ \hline 
    $0$ & $1$ & $-x$ \\ \hline
    $1$ & $x^2$ & $x$ \\ \hline
    \end{tabular} \ \ \ 
    \begin{tabular}{|l|l|l|}
    \hline
    $0$ & $1$ & $0$ \\ \hline 
    $-x$ & $0$ & $1$ \\ \hline
    $x$ & $1$ & $x^2$ \\ \hline
    \end{tabular} \ \ \ 
    \begin{tabular}{|l|l|l|}
    \hline
    $0$ & $0$ & $1$ \\ \hline 
    $1$ & $-x$ & $0$ \\ \hline
    $x^2$ & $x$ & $1$ \\ \hline
    \end{tabular}
$$
for $x=1,\zeta_3,\zeta_3^2$, all of which form valid points in $\Gamma_3(\B)$. (Notice that permuting only the first two columns in the left-most array gives points equivalent to the second array under $x\mapsto x^2$, etc.)

\medskip

\emph{Sub-case 2: $a_1\neq 0$.}
We may assume that $a_1=1$, and so $a_2=0$ (recall that each column must contain a zero coordinate). Similarly to the previous case, if in addition $b_2=0$ then $c_2=1$, and the first non-monomial cubic relation leads to a contradiction ($1=0$); thus assume $b_2=1$. Now the non-monomial cubic relations become $$c_1+c_2=b_1+c_2c_1=1+c_2b_1=0,$$ whose solutions are $c_1=1,\zeta_3,\zeta_3^2$ and $c_2=-c_1,b_1=c_1^2$. We obtain a family of points similar to the one in Sub-case 1, with the second and third rows switched
$$
\begin{tabular}{|l|l|l|}
    \hline
    $1$ & $0$ & $0$ \\ \hline 
    $1$ & $x^2$ & $x$ \\ \hline
    $0$ & $1$ & $-x$ \\ \hline
    \end{tabular} \ \ \ 
    \begin{tabular}{|l|l|l|}
    \hline
    $0$ & $1$ & $0$ \\ \hline 
    $x$ & $1$ & $x^2$ \\ \hline
    $-x$ & $0$ & $1$ \\ \hline
    \end{tabular} \ \ \ 
    \begin{tabular}{|l|l|l|}
    \hline
    $0$ & $0$ & $1$ \\ \hline 
    $x^2$ & $x$ & $1$ \\ \hline
    $1$ & $-x$ & $0$ \\ \hline
    \end{tabular}
$$
for $x=1,\zeta_3,\zeta_3^2$.

\bigskip

\noindent \emph{Case II: $\wt([a_0\colon b_0\colon c_0])=2$.}
Let us assume that $[a_0\colon b_0\colon c_0]=[1\colon b_0\colon 0]$ with $b_0\neq 0$ (and apply symmetries on the columns afterwards, as in Case I).
We now have four sub-cases to consider, corresponding to the conjunction $(a_1=0 \ \vee \ a_2=0) \wedge (b_1=0 \ \vee \ b_2=0)$.

\medskip

\emph{Sub-case 1: $a_1=b_1=0$.} We may assume that $c_1=1$ and so the non-monomial cubic relations become $$ a_2b_0+b_2 = b_2b_0+c_2 = a_2+c_2b_0 = 0. $$ It follows that $a_2\neq 0$ (for otherwise $b_2=c_2=0$ too), so we may set $a_2 = 1$, so the above equations become $$ b_0+b_2 = b_2b_0+c_2 = 1+c_2b_0 = 0 $$ 
and $b_2=-b_0,c_2=b_0^2,b_0^3=-1$, 
whose solutions are $b_0=\zeta_6,\zeta_6^3,\zeta_6^5$. Permuting the columns (recall that we made an initial assumption that $c_0=0$):

$$
\begin{tabular}{|l|l|l|}
    \hline
    $1$ & $x$ & $0$ \\ \hline 
    $0$ & $0$ & $1$ \\ \hline
    $1$ & $-x$ & $x^2$ \\ \hline
    \end{tabular} \ \ \ 
    \begin{tabular}{|l|l|l|}
    \hline
    $0$ & $1$ & $x$ \\ \hline 
    $1$ & $0$ & $0$ \\ \hline
    $x^2$ & $1$ & $-x$ \\ \hline
    \end{tabular} \ \ \ 
    \begin{tabular}{|l|l|l|}
    \hline
    $x$ & $0$ & $1$ \\ \hline 
    $0$ & $1$ & $0$ \\ \hline
    $-x$ & $x^2$ & $1$ \\ \hline
    \end{tabular}
$$
for $x=\zeta_6,\zeta_6^3,\zeta_6^5$. (Notice that permuting only the first two columns in the first array, we get a point equivalent to the first array under $x\mapsto \frac{1}{x}$; permuting the last two columns in the first array, we get a point equivalent to the third array under $x\mapsto \frac{1}{x}$, etc.)

\medskip

\emph{Sub-case 2: $a_1=b_2=0$.} We may further assume that $b_1\neq 0$ (or we are back in the previous sub-case) and thus normalize it to $b_1=1$. The non-monomial cubic relations become $$a_2+a_2c_1b_0+c_2b_0 = a_2b_0 + c_2c_1 = a_2 c_1 + c_2 + c_2 c_1 b_0 = 0.$$ If $a_2=0$ then we may assume that $c_2=1$; from the first equation above, $b_0=0$; from the second, $c_1=0$, and the third one leads to a contradiction ($1=0$). Suppose that $a_2=1$. The above equations become $$1+(c_1+c_2)b_0 = b_0 + c_2c_1 = c_1 + c_2 + c_2 c_1 b_0 = 0.$$ Putting $\alpha = c_1+c_2,\beta = c_1c_2$, we obtain $b_0=-\beta$ and $\alpha \beta = 1, \ \alpha = \beta^2$, so $\beta=1,\zeta_3,\zeta_3^2$. For each such choice of $\beta$ we obtain a quadratic equation (namely, $\lambda^2-\alpha \lambda + \beta = 0$) whose (two distinct) roots are $c_1,c_2$. For $\beta=\zeta_3^i$, the associated quadratic equation is $\lambda^2 - \zeta_3^{2i} \lambda + \zeta_3^i = 0$, whose roots are $\zeta_6,\zeta_6^{-1}$ for $i=0$; $-1,\zeta_6^{-1}$ for $i=1$; and $-1,\zeta_6$ for $i=2$. Recall that these roots are the values taken by $c_1,c_2$. Hence, we obtain $6$ arrays, and by column symmetry,
$$
\begin{tabular}{|l|l|l|}
    \hline
    $1$ & $-x$ & $0$ \\ \hline 
    $0$ & $1$ & $x^{-1}\zeta_6^{\pm 1}$ \\ \hline
    $1$ & $0$ & $x^{-1} \zeta_6^{\mp 1}$ \\ \hline
    \end{tabular} \ \ \ 
\begin{tabular}{|l|l|l|}
    \hline
    $0$ & $1$ & $-x$ \\ \hline 
    $x^{-1}\zeta_6^{\pm 1}$ & $0$ & $1$ \\ \hline
    $x^{-1}\zeta_6^{\mp 1}$ & $1$ & $0$ \\ \hline
    \end{tabular} \ \ \   
\begin{tabular}{|l|l|l|}
    \hline
    $-x$ & $0$ & $1$ \\ \hline 
    $1$ & $x^{-1}\zeta_6^{\pm 1}$ & $0$ \\ \hline
    $0$ & $x^{-1}\zeta_6^{\mp 1}$ & $1$ \\ \hline
    \end{tabular}
$$
for $x=1,\zeta_3,\zeta_3^2$.

\medskip

For the rest of Case II, we may assume that $a_1\neq 0$ and further normalize $a_1=1$. 

\medskip

\emph{Sub-case 3: $a_2=0,b_1=0$.} If $b_2=0$ then we may assume that $c_2=1$, and the first non-monomial cubic relation gives a contradiction ($1=0$). We henceforth assume that further $b_2=1$. Now the non-monomial cubic relations become
\[
b_0+c_1+c_2=c_1 b_0 + c_2 b_0 + c_2 c_1 = 1 + c_2 c_1 b_0 = 0.
\]
Let $\alpha = c_1 + c_2, \beta = c_1 c_2$, so $b_0 = -\alpha$ and $\beta = \alpha^2,\alpha\beta=1$, so $\alpha=1,\zeta_3,\zeta_3^2$ and $\beta = \alpha^{-1}$. As in the previous sub-case, given $\alpha = \zeta_3^i$, we obtain $c_1,c_2$ as the roots of $\lambda^2 - \zeta_3^i \lambda +\zeta_3^{-i}=0$. Similarly to the previous sub-case, the roots are $\zeta_6,\zeta_6^{-1}$ for $i=0$; $-1,\zeta_6$ for $i=1$; and $-1,\zeta_6^{-1}$ for $i=2$. We obtain
$$
\begin{tabular}{|l|l|l|}
    \hline
    $1$ & $-x$ & $0$ \\ \hline 
    $1$ & $0$ & $x \zeta_6^{\pm 1}$ \\ \hline
    $0$ & $1$ & $x \zeta_6^{\mp 1}$ \\ \hline
    \end{tabular} \ \ \ 
\begin{tabular}{|l|l|l|}
    \hline
    $0$ & $1$ & $-x$ \\ \hline 
    $x \zeta_6^{\pm 1}$ & $1$ & $0$ \\ \hline
    $x \zeta_6^{\mp 1}$ & $0$ & $1$ \\ \hline
    \end{tabular} \ \ \   
\begin{tabular}{|l|l|l|}
    \hline
    $-x$ & $0$ & $1$ \\ \hline 
    $0$ & $x \zeta_6^{\pm 1}$ & $1$ \\ \hline
    $1$ & $x \zeta_6^{\mp 1}$ & $0$ \\ \hline
    \end{tabular}
$$
for $x=1,\zeta_3,\zeta_3^2$.

\medskip

\emph{Sub-case 4: $a_2=0,b_2=0$.} We may assume that $c_2=1$. The non-monomial cubic relations become $$ 1+b_1 b_0 = b_0 + c_1 = b_1 + c_1 b_0=0. $$ It follows that $c_1 = -b_0,b_1 = b_0^2,b_0^3+1=0$. The solutions are $b_0=\zeta_6,\zeta_6^3,\zeta_6^5$. We obtain the following arrays:
$$
\begin{tabular}{|l|l|l|}
    \hline
    $1$ & $x$ & $0$ \\ \hline 
    $1$ & $x^2$ & $-x$ \\ \hline
    $0$ & $0$ & $1$ \\ \hline
    \end{tabular} \ \ \ 
\begin{tabular}{|l|l|l|}
    \hline
    $0$ & $1$ & $x$ \\ \hline 
    $-x$ & $1$ & $x^2$ \\ \hline
    $1$ & $0$ & $0$ \\ \hline
    \end{tabular} \ \ \   
\begin{tabular}{|l|l|l|}
    \hline
    $x$ & $0$ & $1$ \\ \hline 
    $x^2$ & $-x$ & $1$ \\ \hline
    $0$ & $1$ & $0$ \\ \hline
    \end{tabular}
$$
for $x=\zeta_6,\zeta_6^3,\zeta_6^5$.

\bigskip

\noindent \emph{Case III: $\wt([a_0\colon b_0\colon c_0])=3$.} We may assume that $a_0=1$.
Since $a_0,b_0,c_0\neq 0$, we have
$(a_1=0 \vee a_2=0)\wedge (b_1=0 \vee b_2=0) \wedge (c_1=0 \vee c_2=0)$. Hence at least one of the second and third rows has weight $1$. Suppose first that the weight of the second row is $1$; up to permuting the columns, we may first assume that $a_1=1,b_1=c_1=0$. Now $a_2=0$, since every column must contain a zero. The non-monomial cubic equations become 
$$ b_2 b_0 + c_2 = b_2 c_0 + c_2 b_0 = b_2 + c_2 c_0 = 0. $$ 
It follows that $b_2\neq 0$ (or else, $c_2=0$ too and the third row is zero), so assume $b_2=1$. Now $c_2=-b_0$ so $c_0=b_0^2$ and $b_0^3=1$, so $b_0=1,\zeta_3,\zeta_3^2$. We obtain the arrays
$$
\begin{tabular}{|l|l|l|}
    \hline
    $1$ & $x$ & $x^2$ \\ \hline 
    $1$ & $0$ & $0$ \\ \hline
    $0$ & $1$ & $-x$ \\ \hline
    \end{tabular} \ \ \ 
\begin{tabular}{|l|l|l|}
    \hline
    $x^2$ & $1$ & $x$ \\ \hline 
    $0$ & $1$ & $0$ \\ \hline
    $-x$ & $0$ & $1$ \\ \hline
    \end{tabular} \ \ \   
\begin{tabular}{|l|l|l|}
    \hline
    $x$ & $x^2$ & $1$ \\ \hline 
    $0$ & $0$ & $1$ \\ \hline
    $1$ & $-x$ & $0$ \\ \hline
    \end{tabular}
$$
for $x=1,\zeta_3,\zeta_3^2$.

We now assume that the weight of the third row is $1$. As above, we may assume that $a_2=1,b_2=0,c_2=0$ and $a_1=0$. The non-monomial cubic relations become $$ b_1 + c_1b_0=b_1b_0+c_1c_0=b_1c_0+c_1=0. $$ If $b_1=0$ then $c_1=0$, so the second row is all zeros, a contradiction. Assume $b_1 = 1$; thus $c_1=-c_0,b_0=c_0^2$ and $c_0^3 = 1$, so $c_0=1,\zeta_3,\zeta_3^2$. We obtain the arrays
$$
\begin{tabular}{|l|l|l|}
    \hline
    $1$ & $x^2$ & $x$ \\ \hline 
    $0$ & $1$ & $-x$ \\ \hline
    $1$ & $0$ & $0$ \\ \hline
    \end{tabular} \ \ \ 
\begin{tabular}{|l|l|l|}
    \hline
    $x$ & $1$ & $x^2$ \\ \hline 
    $-x$ & $0$ & $1$ \\ \hline
    $0$ & $1$ & $0$ \\ \hline
    \end{tabular} \ \ \   
\begin{tabular}{|l|l|l|}
    \hline
    $x^2$ & $x$ & $1$ \\ \hline 
    $1$ & $-x$ & $0$ \\ \hline
    $0$ & $0$ & $1$ \\ \hline
    \end{tabular}
$$
for $x=1,\zeta_3,\zeta_3^2$.

\bigskip

Altogether, we found $90$ points in $\Gamma_3(\B(X,q))$.

\bigskip

Once $\Gamma_3(\B(X,q))$ has been computed, it is easy to compute $\Gamma_4(\B(X,q))$ by brute force (and verify using a straightforward Python code): a point in $\Gamma_4(\B(X,q))$ is represented by an array whose first and last three rows represent points in $\Gamma_3(\B(X,q))$, and such that the whole array satisfies the quartic relation we have computed.

We found $36$ points, partitioned into $7$ orbits of $\mathbb{S}_3$ (through its action on the rack $X$); representatives of these orbits are brought below (notice that the second and last arrays have stabilizers of order $2$).

\[
\begin{array}{|c|c|c|}
\hline
1 & \zeta_6 & 0\\
\hline
0 & 0 & 1\\
\hline
1 & \zeta_3^2 & \zeta_3\\
\hline
1 & \zeta_6^{-1} & 0\\
\hline
\end{array}
\qquad
\begin{array}{|c|c|c|}
\hline
1 & -1 & 0\\
\hline
0 & 0 & 1\\
\hline
1 & 1 & 1\\
\hline
1 & -1 & 0\\
\hline
\end{array}
\qquad
\begin{array}{|c|c|c|}
\hline
1 & -1 & 0\\
\hline
0 & 1 & \zeta_6\\
\hline
1 & 0 & \zeta_6^{-1}\\
\hline
-1 & 1 & 0\\
\hline
\end{array}
\qquad
\begin{array}{|c|c|c|}
\hline
1 & \zeta_6^{-1} & 0\\
\hline
0 & 1 & \zeta_6^{-1}\\
\hline
1 & 0 & -1\\
\hline
\zeta_6^{-1} & 1 & 0\\
\hline
\end{array}
\]

\[
\begin{array}{|c|c|c|}
\hline
1 & \zeta_6 & 0\\
\hline
0 & 1 & -1\\
\hline
1 & 0 & \zeta_6\\
\hline
\zeta_6 & 1 & 0\\
\hline
\end{array}
\qquad
\begin{array}{|c|c|c|}
\hline
1 & \zeta_6 & 0\\
\hline
1 & \zeta_3 & \zeta_3^2\\
\hline
0 & 0 & 1\\
\hline
1 & \zeta_6^{-1} & 0\\
\hline
\end{array}
\qquad
\begin{array}{|c|c|c|}
\hline
1 & -1 & 0\\
\hline
1 & 1 & 1\\
\hline
0 & 0 & 1\\
\hline
1 & -1 & 0\\
\hline
\end{array}
\]


Finally, it is again straightforward to check which points lie in $\Gamma_5(\B(X,q))$: a $5$-row array representing a point in $\Gamma_5(\B(X,q))$ must have its first and last four rows representing points in $\Gamma_4(\B(X,q))$. This can be done by brute force, leading to the finding that $\Gamma_5(\B(X,q))=\emptyset$.
\end{proof}

This leads to the following problem, specifying Problem \ref{prob:main} to the simplest possible underlying rack:

\begin{prob}
Let $X=\Aff(3,-1)$, the unique indecomposable rack of cardinality $3$. Consider $q\equiv \zeta_m$. Describe $\Gamma_d(\B(X,q))$.
\end{prob}

\section{Thick points over \texorpdfstring{$\E_3$}{E3}} 
\label{sec:thick}

As demonstrated in the current paper, (truncated) point modules over Nichols algebras tend to be very restricted. Subsequently, it might be hard to utilize point spaces to prove largeness properties (such as infinite-dimensionality) of Nichols algebras. Indeed, for Fomin--Kirillov algebras, there are no truncated point modules of degree greater than one \cite{GMSW}.

In this section, we propose that instead of studying (truncated) point modules of Nichols algebras, one might want to consider (truncated) \emph{thick} point modules\footnote{We are grateful to James J. Zhang for proposing to us this idea.}. These can reflect a richer portion of the geometric data of the underlying algebras, as we demonstrate here for the smallest Nichols algebra over a non-abelian group: the Fomin--Kirillov algebra $\E_3$.

A thick (truncated) point module of width $c$ and of degree $d$ is a graded module $M=M_0\oplus\cdots\oplus M_d$ generated in degree zero with Hilbert polynomial $H_M(t)=c+ct+\cdots+ct^d$. We denote the parametrizing space of such modules by $\Gamma_{c+ct+\cdots+ct^d}$ or simply by $\Gamma_{c,d}$. Thus $\Gamma_{1,d}$ is the space of truncated $d$-point modules.


Fix a presentation $$\E_3=K\left<a,b,c\right>/\left<a^2,b^2,c^2,ca-ab-bc,cb-ac+ba\right>$$ (under $a=x_{12},b=x_{13},c=x_{23}$ with the usual generators of Fomin--Kirillov algebras). We compute a homogeneous basis for $\E_3$:
\begin{align} \label{Basis}
1,a,b,c,ab,ac,ba,bc,aba,abc,bac,abac.  
\end{align}


\begin{lem} \label{lem:Possible Hilbert Polynomials}
Let $M$ be a thick point module of width $2$ and degree $2$ over $\E_3$. Then the only possible Hilbert polynomials of proper submodules of $M$ that are generated in degree zero are $1$ and $1+2t+2t^2$.
\end{lem}
\begin{proof}
A priori, the possible Hilbert polynomials are:
$$ 1,\ 1+t,\ 1+t+t^2,\ 1+t+2t^2,\ 1+2t+t^2,\ 1+2t+2t^2. $$
Notice that a proper graded submodule of $M$ generated in degree zero must be cyclic, hence the constant term in its Hilbert polynomial must be $1$; and that if a submodule contains $M_1$ then it must contain $M_2$, as $M$ is generated in degree zero, hence $1+2t+t^2$ is impossible.

By \cite{GMSW}, the Hilbert polynomial $1+t+t^2$ is eliminated. If $\E_3$ had a module generated in degree zero with Hilbert polynomial $1+t+2t^2$, it would further admit a quotient (dividing out by an arbitrary degree $2$ element) with Hilbert polynomial $1+t+t^2$, a contradiction.
If $N\leq M$ were a module with Hilbert polynomial $1+t$ then $M/N$ would have Hilbert polynomial $1+t+2t^2$ (still generated in degree zero).
We remain with the Hilbert polynomials $1$ and $1+2t+2t^2$.
\end{proof}

In what follows, we denote the left ideal generated by an element $a$ in an algebra $A$ by $(a)$.

\begin{lem}
\label{lem:cases}
    If $M$ is a graded module generated in degree zero 
    with Hilbert polynomial $1+2t+2t^2$ then $M \cong \left(\E_3/(\alpha a+\beta b+\gamma c)\right)_{\leq 2}$ 
    where: $$[\alpha\colon\beta\colon\gamma]\in \{[1\colon 0\colon 0],[0\colon 1\colon 0],[0\colon 0\colon 1]\} \cup \{[1\colon \theta\colon \theta^2]\ |\ \theta^3=-1\}.$$
\end{lem}

\begin{proof}
Since $M$ is cyclic, it must be isomorphic to a quotient of $\E_3$ by a homogeneous left ideal $L\leq \E_3$. The degree one part of $L$ must be one-dimensional, since $\dim_K(\E_3)_1 - \dim_K(L)_1 = \dim_K M_1=2$, so $M$ is a quotient of $\E_3/(\xi)$ for some $\xi = \alpha a+\beta b+\gamma c$, and not all $\alpha,\beta,\gamma$ are zero. 

We compute a basis for $\left(\E_3/(\xi)\right)_2$ as a $K$-vector space. 
Notice that $(\xi)_2=\Span_K\{a\xi,b\xi,c\xi\}$, where:
\begin{eqnarray*}
    a\xi & = & \beta ab + \gamma ac \\
    b\xi & = & \alpha ba + \gamma bc \\
    c\xi & = & \alpha ca + \beta cb = \alpha (ab+bc) + \beta (ac-ba)
\end{eqnarray*}
and with respect to the basis \eqref{Basis}, ordered $ab,ac,ba,bc$, $(\xi)_2$ is spanned by the columns of the matrix
\[
\Lambda_{\xi} =  \begin{pmatrix} 
    \beta & 0 & \alpha \\
    \gamma & 0 & \beta \\
    0 & \alpha & -\beta \\
    0 & \gamma & \alpha
\end{pmatrix}.
\]
Therefore, $\dim_K (\xi)_2 \in \{2,3\}$, and so $\dim_K \left(\E_3/(\xi)\right)_2 \in \{2,1\}$ (respectively). Since $\E_3/(\xi)\twoheadrightarrow M$ and $M_2$ is two-dimensional, it follows that $M\cong \left(\E_3/(\xi)\right)_{\leq 2}$. 

The possible linear forms $\xi$ for which $\left(\E_3/(\xi)\right)_{\leq 2}$ has Hilbert series $1+2t+2t^2$ (rather than $1+2t+t^2$) are those for which $\rank(\Lambda_\xi)=2$. Assume first that $\gamma\neq 0$. Then the first and second columns are linearly independent, so we must check when the third is a linear combination of them. This happens if and only if
\[
\left( \begin{matrix} \alpha \\ \beta \\ -\beta \\ \alpha \end{matrix} \right) = \left( \begin{matrix} \lambda \beta \\ \lambda \gamma \\ \mu \alpha \\ \mu \gamma \end{matrix} \right)
\]
for some $\lambda,\mu \in K$; namely, either $\alpha=\beta=0$, or else, $\alpha,\beta\neq 0$ and $\frac{\alpha}{\beta} = \lambda = \frac{\beta}{\gamma}$ and $-\frac{\beta}{\alpha}=\mu=\frac{\alpha}{\gamma}$. The latter case is equivalent to  $\beta^2=\alpha\gamma,\ \alpha^2=-\beta\gamma$ (while $\alpha,\beta,\gamma\neq 0$). This implies that $\beta^3=-\alpha^3$ (and $\gamma$ is uniquely determined by $\alpha,\beta$ in this case, namely, $\gamma = \frac{\beta^2}{\alpha} = -\frac{\alpha^2}{\beta}$). Therefore, when written in projective coordinates in $\mathbb{P}^2$, $\xi$ takes the form $[0\colon 0\colon 1]$ or $[1\colon \theta\colon \theta^2]$ where $\theta^3=-1$, namely, $\theta\in \{\zeta_6^{\pm 1},-1\}$. 

If $\gamma=0$ then the three columns of $\Lambda_\xi$ are linearly independent unless either $\alpha=0$ or $\beta=0$. These correspond to $\xi$ taking the form $[1\colon 0\colon 0]$ or $[0\colon 1\colon 0]$. Therefore, $M\cong \left(\E_3/(\xi)\right)_{\leq 2}$ for $\xi$ as in the assertion.

Finally, notice that in $\left(\E_3/(\xi)\right)_{\leq 2}$, the degree-one part of the annihilator of any non-zero degree zero element is $(\xi)_1$, so 
$\left(\E_3/(\xi)\right)_{\leq 2}\cong \left(\E_3/(\xi')\right)_{\leq 2}$ implies that $(\xi)_1=(\xi')_1$, which implies that $\xi=\lambda \xi'$ for some $\lambda\in K^{\times}$, so $\xi,\xi'$ represent the same point in $\mathbb{P}^2$.
\end{proof}

\begin{proof}[Proof of Theorem \ref{thm:22}]
Let $M=M_0\oplus M_1\oplus M_2$ be a module over $\E_3$, generated by $M_0$, with Hilbert polynomial $2+2t+2t^2$. We claim that $$M\cong \left(\E_3/(\xi)\right)_{\leq 2} \oplus K$$ where $K$ is the trivial module (in degree zero) and $\xi=\alpha a+\beta b+\gamma c$ where $$[\alpha\colon\beta\colon\gamma]\in \{[1\colon 0\colon 0],[0\colon 1\colon 0],[0\colon 0\colon 1]\} \cup \{[1\colon \theta\colon \theta^2]\ |\ \theta^3=-1\}.$$
Fix a basis $M_0=Ke_1 \oplus Ke_2$. 
Consider the graded cyclic submodules $\E_3\cdot e_1,\E_3\cdot e_2$. By Lemma \ref{lem:Possible Hilbert Polynomials}, their Hilbert polynomials are either $1$ or $1+2t+2t^2$. Since $M=\E_3\cdot e_1 + \E_3 \cdot e_2$, it is impossible that both have Hilbert polynomial $1$, so without loss of generality, let us assume that $H_{\E_3\cdot e_1}(t) = 1+2t+2t^2$.

By Lemma \ref{lem:cases}, it follows that $\E_3\cdot e_1\cong \left(\E_3/(\xi)\right)_{\leq 2}$ where $$[\alpha\colon\beta\colon\gamma]\in \{[1\colon 0\colon 0],[0\colon 1\colon 0],[0\colon 0\colon 1]\} \cup \{[1\colon \theta\colon \theta^2]\ |\ \theta^3=-1\}.$$


Suppose $\gamma\neq 0$. Then $ab,ba$ are linearly independent modulo $(\xi)$, since then any linear combination of the columns of $\Lambda_\xi$ from Lemma \ref{lem:cases} takes the form
\[
\left( \begin{matrix} \lambda \beta \\ \lambda \gamma \\ \mu \alpha \\ \mu \gamma \end{matrix} \right)
\]
and cannot have only the first and the third entries (these correspond to $ab,ba$) non-zero. Therefore, the following is a graded basis of $M$:
\[
M_0=\Span_K\{e_1,e_2\},\ M_1=\Span_K\{a\cdot e_1,b\cdot e_1\},\ M_2=\Span_K\{ab\cdot e_1,ba\cdot e_1\}.
\]
The Hilbert polynomial of the quotient module $M/\left(\E_3\cdot e_1\right)$ is $2+2t+2t^2-(1+2t+2t^2)=1$, so it is isomorphic to the trivial module $K$ in degree zero; hence $a\cdot e_2,b\cdot e_2,c\cdot e_2\in \Span_K\{a\cdot e_1,b\cdot e_1\}$. Write:
\begin{align}
\label{eq:1}a\cdot e_2 & = (\lambda_1 a + \mu_1 b)\cdot e_1, \\
\label{eq:2}b\cdot e_2 & = (\lambda_2 a + \mu_2 b)\cdot e_1, \\
\label{eq:3}c\cdot e_2 & = (\lambda_3 a + \mu_3 b)\cdot e_1
\end{align}
for some scalars $\lambda_i,\mu_i,1\leq i\leq 3$. 
Multiplying \eqref{eq:1} by $a$ and \eqref{eq:2} 
by $b$, we get that $\mu_1=0$ and $\lambda_2=0$, respectively.
Multiplying \eqref{eq:1} by $c$ we get $ca\cdot e_2=\lambda_1ca\cdot e_1$. Hence: 
\begin{equation} 
\label{eq:ca} 
\begin{aligned} 
ca\cdot e_2 &= \lambda_1ca\cdot e_1 = \lambda_1(ab+bc)\cdot e_1 \\ 
& = \lambda_1 ab\cdot e_1 + \lambda_1 b\left(-\frac{\alpha}{\gamma} a - \frac{\beta}{\gamma} b\right) \cdot e_1 = \left(\lambda_1 ab - \lambda_1\frac{\alpha}{\gamma}ba\right)\cdot e_1.
\end{aligned} 
\end{equation}
Multiplying \eqref{eq:1} by $b$ we get:
$$ ba \cdot e_2 = \lambda_1 ba\cdot e_1. $$
Multiplying \eqref{eq:2} by $c$, we get 
\begin{equation}
\label{eq:cb}
\begin{aligned}
    cb\cdot e_2 &= \mu_2 cb\cdot e_1 = \mu_2 (ac-ba)\cdot e_1 \\ 
    &= \mu_2 a\left(-\frac{\alpha}{\gamma} a-\frac{\beta}{\gamma} b\right)\cdot e_1 - \mu_2 ba \cdot e_1 = \left(-\mu_2\frac{\beta}{\gamma} ab -\mu_2 ba\right)\cdot e_1.
\end{aligned}
\end{equation}
Similarly, multiplying \eqref{eq:2} by $a$, we get:
$$ ab\cdot e_2 = \mu_2ab\cdot e_1. $$ 
Now multiplying \eqref{eq:3} by $a$ and $b$ we get
$$ ac\cdot e_2 = \mu_3 ab\cdot e_1,\ \ bc\cdot e_2 = \lambda_3 ba\cdot e_1,$$
respectively. 
Thus by \eqref{eq:ca}:
$$ \lambda_1 ab\cdot e_1  
- \lambda_1\frac{\alpha}{\gamma}ba\cdot e_1 = ca\cdot e_2 = (ab+bc)\cdot e_2 = \mu_2 ab\cdot e_1 + \lambda_3 ba\cdot e_1  $$
so $\lambda_1 = \mu_2$ and $\lambda_3 = -\lambda_1\frac{\alpha}{\gamma}$. 
By \eqref{eq:cb}:
$$ -\mu_2\frac{\beta}{\gamma} ab\cdot e_1 - \mu_2 ba\cdot e_1 = cb\cdot e_2 = (ac-ba)\cdot e_2 = \mu_3ab\cdot e_1 - \lambda_1 ba \cdot e_1 $$
so $-\mu_2\frac{\beta}{\gamma} = \mu_3$.
Substituting into \eqref{eq:1},\eqref{eq:2},\eqref{eq:3} we obtain:
\begin{eqnarray*} 
 a\cdot e_2 & = &  \lambda_1 a\cdot e_1,\\ b\cdot e_2 & = & \mu_2 b\cdot e_1=\lambda_1 b\cdot e_1, \\  c\cdot e_2 
 & = & \lambda_3 a\cdot e_1 + \mu_3 b\cdot e_1 = -\lambda_1\frac{\alpha}{\gamma}a\cdot e_1 - \mu_2\frac{\beta}{\gamma}b\cdot e_1 \\ & = & -\lambda_1\frac{\alpha}{\gamma}a\cdot e_1 - \lambda_1 \frac{\beta}{\gamma}b\cdot e_1 = \lambda_1c\cdot e_1
\end{eqnarray*}
so $$ a\cdot (e_2-\lambda_1 e_1)=b\cdot (e_2-\lambda_1 e_1)=c\cdot(e_2-\lambda_1 e_1)=0. $$ It follows that $\E_3\cdot (e_2-\lambda_1 e_1) \cong K\oplus 0 \oplus 0$ and thus $$ M=\E_3\cdot e_1 \oplus \E_3\cdot (e_2-\lambda_1 e_1)\cong \left(\E_3/(\xi)\right)_{\leq 2} \oplus K, $$ as claimed.

Recall that we made the assumption that $\gamma\neq 0$. We have an automorphism $\sigma\colon \E_3\rightarrow \E_3$ given by $\sigma(a)=b,\sigma(b)=c,\sigma(c)=-a$. 

Suppose that $M=M_0\oplus M_1\oplus M_2$ is an $\E_3$-module generated by $M_0$ with Hilbert polynomial $2+2t+2t^2$. Recall that we established that $M$ admits a degree zero element, say $e_1$, such that $\E_3\cdot e_1 \cong \left(\E_3/(\xi)\right)_{\leq 2}$ with $\xi=\alpha a+\beta b+\gamma c$ ($\xi\neq 0$). Suppose that $\gamma=0$ but $\alpha\neq 0$ or $\beta\neq 0$. Given a graded automorphism $\tau$ of $\E_3$, the module $\tau^*M$ has underlying vector space $M_0\oplus M_1\oplus M_2$ with twisted action $x\bullet m=\tau(x)\cdot m$. 
Now for $\tau=\sigma^{\pm 1}$ we see that in $\tau^*M$, 
\[
\E_3\bullet e_1\cong \left(\E_3/(\alpha' a+\beta' b+\gamma' c)\right)_{\leq 2}
\]
with $\gamma'\neq 0$. Hence the rest of the argument given above applies to $\tau^*M$, showing that it is isomorphic to $N \oplus K$ where $N$ is generated in degree $0$ and has Hilbert polynomial $1+2t+2t^2$, and $K$ is a trivial module in degree zero. Hence $$ M\cong {\tau^{-1}}^* \tau^* M\cong {\tau^{-1}}^* N \oplus {\tau^{-1}}^* K, $$ and since twisting preserves Hilbert polynomials, the result follows in any case.

Finally, notice that if $$ \left(\E_3/(\xi)\right)_{\leq 2} \oplus K \cong \left(\E_3/(\xi')\right)_{\leq 2} \oplus K $$ then any isomorphism between them must map the (unique) trivial degree zero submodule on the left hand side to that of the right hand side. Therefore $$ \left(\E_3/(\xi)\right)_{\leq 2} \cong \left(\E_3/(\xi')\right)_{\leq 2} $$
and, as mentioned before (see the end of the proof of Lemma \ref{lem:cases}), it follows that $\xi=\lambda \xi'$ for some $\lambda\in K^\times$. 
It follows that there is a bijective correspondence between $\Gamma_{2,2}(\E_3)$ and the set
$$\{[1\colon 0\colon 0],[0\colon 1\colon 0],[0\colon 0\colon 1]\} \cup \{[1\colon \theta\colon \theta^2]\ |\ \theta^3=-1\}.$$ 
If $\charac(K)=3$ then we only have four points: $[1\colon 0\colon 0],[0\colon 1\colon 0],[0\colon 0\colon 1],[1\colon -1\colon 1]$. If $\charac(K)\neq 3$, we have only four points if $K$ does not contain a primitive cubic root of unity, and six points otherwise (as $\theta$ ranges over $\{-1,\zeta_6,\zeta_6^{-1}\}$).


Finally, we conclude that there are no thick point modules over $\E_3$ of width $2$ of degree $>2$.  
Suppose that $M=M_0\oplus\cdots\oplus M_d$ ($d\geq 3$) is an $\E_3$-module generated in degree zero and with $\dim_K M_i=2$ for $0\leq i\leq 2$. Let $M_{\leq 2}=M_0\oplus M_1\oplus M_2$ be its degree-$2$ truncation. As we proved above, $M_{\leq 2}\cong \left(\E_3/(\xi)\right)_{\leq 2}\oplus K$. Fix generators $e_1,e_2\in M_0$ such that $\xi\cdot e_1=0$ and $a\cdot e_2,b\cdot e_2,c\cdot e_2=0$. It follows that $M$ is a quotient module of $\E_3/(\xi)\oplus K$ (here $K$ is a trivial module in degree zero). Recall from \eqref{Basis} that $(\E_3)_3$ is spanned by $aba,abc,bac$. As noted before, we may assume that $\gamma\neq 0$, so $abc,bac$ belong to $\Span_K\{aba\}$ modulo $(\xi)$. It follows that $\dim_K M_3 \leq \dim_K(\E_3/(\xi))_3\leq 1$, so $H_M(t)\neq 2+2t+2t^2+2t^3$.
\end{proof}

Theorem \ref{thm:22} suggests the following broader problem:

\begin{prob}
Describe the truncated thick point modules over Fomin--Kirillov algebras.
\end{prob}

 \bibliographystyle{abbrv}
 \bibliography{refs}

\begin{thebibliography}{10}

\bibitem{An2004}
N.~Andruskiewitsch.
\newblock Some remarks on {Nichols} algebras.
\newblock In {\em Hopf algebras. Proceedings from the international conference, DePaul University, Chicago, IL, USA held during the 2001--2002 academic year}, pages 35--45. New York, NY: Marcel Dekker, 2004.

\bibitem{zbMATH06797653}
N.~Andruskiewitsch.
\newblock On finite-dimensional {Hopf} algebras.
\newblock In {\em Proceedings of the international congress of mathematicians (ICM 2014), Seoul, Korea, August 13--21, 2014. Vol. II: Invited lectures}, pages 117--141. Seoul: KM Kyung Moon Sa, 2014.

\bibitem{ABFF2023}
N.~Andruskiewitsch, D.~Bagio, S.~D. Flora, and D.~Flores.
\newblock On the {Laistrygonian} {Nichols} algebras that are domains.
\newblock {\em J. Algebr. Comb.}, 58(2):549--568, 2023.

\bibitem{zbMATH07796386}
N.~Andruskiewitsch, G.~Carnovale, and G.~A. Garc{\'{\i}}a.
\newblock Finite-dimensional pointed {Hopf} algebras over finite simple groups of {Lie} type {VII}. {Semisimple} classes in {{\(\mathbf{PSL}_n (q)\)}} and {{\(\mathbf{PSp}_{2n}(q)\)}}.
\newblock {\em J. Algebra}, 639:354--397, 2024.

\bibitem{zbMATH05896174}
N.~Andruskiewitsch, F.~Fantino, G.~A. Garc{\'{\i}}a, and L.~Vendramin.
\newblock On {Nichols} algebras associated to simple racks.
\newblock In {\em Groups, algebras and applications. XVIII Latin American algebra colloquium, S\~ao Pedro, Brazil, August 3--8, 2009. Proceedings.}, pages 31--56. Providence, RI: American Mathematical Society (AMS), 2011.

\bibitem{zbMATH05903856}
N.~Andruskiewitsch, F.~Fantino, M.~Gra{\~n}a, and L.~Vendramin.
\newblock Finite-dimensional pointed {Hopf} algebras with alternating groups are trivial.
\newblock {\em Ann. Mat. Pura Appl. (4)}, 190(2):225--245, 2011.

\bibitem{zbMATH05869018}
N.~Andruskiewitsch, F.~Fantino, M.~Gra{\~n}a, and L.~Vendramin.
\newblock Pointed {Hopf} algebras over the sporadic simple groups.
\newblock {\em J. Algebra}, 325(1):305--320, 2011.

\bibitem{Andruskiewitsch2003}
N.~Andruskiewitsch and M.~Gra{\~n}a.
\newblock From racks to pointed {Hopf} algebras.
\newblock {\em Adv. Math.}, 178(2):177--243, 2003.

\bibitem{arXiv:2411.02304}
N.~Andruskiewitsch, I.~Heckenberger, and L.~Vendramin.
\newblock Pointed {Hopf} algebras of odd dimension and {Nichols} algebras over solvable groups.
\newblock Preprint, {arXiv}:2411.02304 [math.{QA}] (2024), 2024.

\bibitem{zbMATH01247621}
N.~Andruskiewitsch and H.-J. Schneider.
\newblock Lifting of quantum linear spaces and pointed {Hopf} algebras of order {{\(p^3\)}}.
\newblock {\em J. Algebra}, 209(2):658--691, 1998.

\bibitem{zbMATH06224213}
I.~Angiono.
\newblock On {Nichols} algebras of diagonal type.
\newblock {\em J. Reine Angew. Math.}, 683:189--251, 2013.

\bibitem{ATV}
M.~Artin, J.~Tate, and M.~Van~den Bergh.
\newblock Some algebras associated to automorphisms of elliptic curves.
\newblock The {Grothendieck} {Festschrift}, {Collect}. {Artic}. in {Honor} of the 60th {Birthday} of {A}. {Grothendieck}. {Vol}. {I}, {Prog}. {Math}. 86, 33-85 (1990)., 1990.

\bibitem{ATV2}
M.~Artin, J.~Tate, and M.~Van~den Bergh.
\newblock Modules over regular algebras of dimension 3.
\newblock {\em Invent. Math.}, 106(2):335--388, 1991.

\bibitem{AZ2}
M.~Artin and J.~J. Zhang.
\newblock Noncommutative projective schemes.
\newblock {\em Adv. Math.}, 109(2):228--287, 1994.

\bibitem{AZ}
M.~Artin and J.~J. Zhang.
\newblock Abstract {Hilbert} schemes. {I}.
\newblock {\em Algebr. Represent. Theory}, 4(4):305--394, 2001.

\bibitem{BG}
J.~P. Bell and B.~Greenfeld.
\newblock Noncommutative point spaces of symbolic dynamical systems.
\newblock {\em Adv. Math.}, 469:50, 2025.
\newblock Id/No 110211.

\bibitem{zbMATH01077111}
W.~Bosma, J.~Cannon, and C.~Playoust.
\newblock The {Magma} algebra system. {I}: {The} user language.
\newblock {\em J. Symb. Comput.}, 24(3-4):235--265, 1997.

\bibitem{Chan_fat}
D.~Chan.
\newblock Twisted rings and moduli stacks of ``fat'' point modules in non-commutative projective geometry.
\newblock {\em Adv. Math.}, 229(4):2184--2209, 2012.

\bibitem{zbMATH05942621}
G.~A. Garc{\'{\i}}a and A.~Garc{\'{\i}}a~Iglesias.
\newblock Finite-dimensional pointed {Hopf} algebras over {{\(\mathbb S_4\)}}.
\newblock {\em Isr. J. Math.}, 183:417--444, 2011.

\bibitem{Goetz2007}
P.~Goetz.
\newblock Fat point modules over generalized {Laurent} polynomial rings.
\newblock {\em J. Algebra}, 313(2):657--671, 2007.

\bibitem{zbMATH02100251}
M.~Gra{\~n}a.
\newblock Indecomposable racks of order {{\(p^2\)}}.
\newblock {\em Beitr. Algebra Geom.}, 45(2):665--676, 2004.

\bibitem{zbMATH05918254}
M.~Gra{\~n}a, I.~Heckenberger, and L.~Vendramin.
\newblock Nichols algebras of group type with many quadratic relations.
\newblock {\em Adv. Math.}, 227(5):1956--1989, 2011.

\bibitem{GMSW}
B.~Greenfeld, S.~Mathison, A.~Saini, and S.~Wynn.
\newblock Modules over {Fomin}-{Kirillov} algebras and their subalgebras.
\newblock Preprint, {arXiv}:2410.16578 [math.{RA}] (2024), 2024.

\bibitem{zbMATH05027328}
I.~Heckenberger.
\newblock The {Weyl} groupoid of a {Nichols} algebra of diagonal type.
\newblock {\em Invent. Math.}, 164(1):175--188, 2006.

\bibitem{zbMATH05376870}
I.~Heckenberger.
\newblock Classification of arithmetic root systems.
\newblock {\em Adv. Math.}, 220(1):59--124, 2009.

\bibitem{zbMATH06050325}
I.~Heckenberger, A.~Lochmann, and L.~Vendramin.
\newblock Braided racks, {Hurwitz} actions and {Nichols} algebras with many cubic relations.
\newblock {\em Transform. Groups}, 17(1):157--194, 2012.

\bibitem{zbMATH07837859}
I.~Heckenberger, E.~Meir, and L.~Vendramin.
\newblock Finite-dimensional {Nichols} algebras of simple {Yetter}-{Drinfeld} modules (over groups) of prime dimension.
\newblock {\em Adv. Math.}, 444:30, 2024.
\newblock Id/No 109637.

\bibitem{zbMATH06694062}
I.~Heckenberger and L.~Vendramin.
\newblock A classification of {Nichols} algebras of semisimple {Yetter}-{Drinfeld} modules over non-abelian groups.
\newblock {\em J. Eur. Math. Soc. (JEMS)}, 19(2):299--356, 2017.

\bibitem{zbMATH06736624}
I.~Heckenberger and L.~Vendramin.
\newblock The classification of {Nichols} algebras over groups with finite root system of rank two.
\newblock {\em J. Eur. Math. Soc. (JEMS)}, 19(7):1977--2017, 2017.

\bibitem{RRZ}
Z.~Reichstein, D.~Rogalski, and J.~J. Zhang.
\newblock Projectively simple rings.
\newblock {\em Adv. Math.}, 203(2):365--407, 2006.

\bibitem{Rogalski}
D.~Rogalski.
\newblock Noncommutative projective geometry.
\newblock In {\em Noncommutative algebraic geometry. Lecture notes based on courses given at the Summer Graduate School at the Mathematical Sciences Research Institute (MSRI), Berkeley, CA, USA, June 2012}, pages 13--70. Cambridge: Cambridge University Press, 2016.

\bibitem{RZ}
D.~Rogalski and J.~J. Zhang.
\newblock Canonical maps to twisted rings.
\newblock {\em Math. Z.}, 259(2):433--455, 2008.

\bibitem{SheltonVancliff2002a}
B.~Shelton and M.~Vancliff.
\newblock Schemes of line modules. {I}.
\newblock {\em J. Lond. Math. Soc. (2)}, 65(3):575--590, 2002.

\bibitem{SheltonVancliff2002b}
B.~Shelton and M.~Vancliff.
\newblock Schemes of line modules. {II}.
\newblock {\em Comm. Algebra}, 30(5):2535--2552, 2002.

\bibitem{SW}
S.~J. Sierra and C.~Walton.
\newblock The universal enveloping algebra of the {Witt} algebra is not {Noetherian}.
\newblock {\em Adv. Math.}, 262:239--260, 2014.

\bibitem{Smith}
S.~P. Smith.
\newblock The space of {Penrose} tilings and the noncommutative curve with homogeneous coordinate ring {{\(k{{\langle}} x,y {{\mathrm{ran}gle}}/(y^2)\)}}.
\newblock {\em J. Noncommut. Geom.}, 8(2):541--586, 2014.

\bibitem{Tao}
T.~Tao.
\newblock An uncertainty principle for cyclic groups of prime order.
\newblock {\em Math. Res. Lett.}, 12(1):121--127, 2005.

\bibitem{Vancliff2024}
M.~Vancliff.
\newblock An example of a quadratic {AS}-regular algebra without any point modules.
\newblock In {\em Recent advances in noncommutative algebra and geometry. Conference on recent advances and new directions in the interplay of noncommutative algebra and geometry in honor of S. Paul Smith on the occasion of his 65th birthday, University of Washington, Seattle, WA, June 20--24, 2022}, pages 255--259. Providence, RI: American Mathematical Society (AMS), 2024.

\bibitem{zbMATH06204244}
L.~Vendramin.
\newblock Nichols algebras associated to the transpositions of the symmetric group are twist-equivalent.
\newblock {\em Proc. Am. Math. Soc.}, 140(11):3715--3723, 2012.

\end{thebibliography}

\end{document}